\pdfoutput=1
\documentclass{amsart}
\usepackage{amsthm, amssymb,amsmath,amsfonts,amscd, mathrsfs,mathtools,color, enumerate,fullpage, thmtools, tikz-cd, comment}

\usepackage[pagebackref]{hyperref}
\usepackage[capitalize, noabbrev]{cleveref}
\usepackage[all]{xy}
\usepackage{latexsym}

\newtheorem{definition}{Definition}[section]
\newtheorem{lemma}[definition]{Lemma}
\newtheorem{proposition}[definition]{Proposition}

\newtheorem{theorem}[definition]{Theorem}
\newtheorem{conjecture}[definition]{Conjecture}
\newtheorem*{notation}{Notation}

\theoremstyle{remark}

\declaretheorem[name=Remark,sibling=theorem,qed={\lower-0.3ex\hbox{$\diamond$}}]{remark}
\declaretheorem[name=Note,sibling=theorem,qed={\lower-0.3ex\hbox{$\diamond$}}]{note}

\title{On $p$-adic $L$-functions for $\GSp_4 \times \GL_2$}

\newcommand{\FL}{FL}
\DeclareMathOperator{\Iw}{Iw}
\newcommand{\ZZ}{\mathbf{Z}}
\newcommand{\QQ}{\mathbf{Q}}

\newcommand{\CC}{\mathbf{C}}
\newcommand{\FF}{\mathbf{F}}
\newcommand{\GG}{\mathbf{G}}

\newcommand{\DD}{\mathbf{D}}

\renewcommand{\AA}{\mathbf{A}}
\newcommand{\Qp}{\QQ_p}

\newcommand{\Zp}{\ZZ_p}

\newcommand{\cO}{\mathcal{O}}

\newcommand{\cW}{\mathcal{W}}

\newcommand{\hl}{\pmb{\ell}}

\newcommand{\hr}{{\mathbf r}}
\newcommand{\htt}{{\mathbf t}}
\newcommand{\hz}{\mathbf{z}}

\newcommand{\fp}{\mathfrak{p}}

\newcommand{\cX}{\mathcal{X}}
\newcommand{\cY}{\mathcal{Y}}

\newcommand{\cB}{\mathcal{B}}

\renewcommand{\ge}{\geqslant}
\renewcommand{\le}{\leqslant}

\DeclareMathOperator{\ad}{ad}

\DeclareMathOperator{\an}{an}
\DeclareMathOperator{\anInd}{anInd}

\DeclareMathOperator{\cris}{cris}

\DeclareMathOperator{\cusp}{cusp}

\DeclareMathOperator{\diag}{diag}

\DeclareMathOperator{\ES}{ES}

\DeclareMathOperator{\fin}{f}

\DeclareMathOperator{\Frob}{Frob}
\DeclareMathOperator{\fs}{fs}
\DeclareMathOperator{\Gal}{Gal}
\DeclareMathOperator{\GL}{GL}
\DeclareMathOperator{\Gr}{Gr}
\DeclareMathOperator{\GSp}{GSp}

\DeclareMathOperator{\HT}{HT}
\DeclareMathOperator{\id}{id}
\DeclareMathOperator{\imp}{imp}

\DeclareMathOperator{\Kl}{Kl}

\DeclareMathOperator{\nc}{nc}

\DeclareMathOperator{\PR}{PR}
\DeclareMathOperator{\pr}{pr}

\DeclareMathOperator{\Si}{Si}

\DeclareMathOperator{\Spa}{Spa}

\DeclareMathOperator{\Spec}{Spec}

\DeclareMathOperator{\ssa}{ss}
\DeclareMathOperator{\sss}{sss}

\DeclareMathOperator{\tor}{tor}

\newcommand{\tU}{\mathtt{U}}
\newcommand{\tZ}{\mathtt{Z}}

\newenvironment{smatrix}{\left(\begin{smallmatrix}}{\end{smallmatrix}\right)}

\usepackage{tabto}

\newcommand{\tbt}[4]{
	\ensuremath{
		\begin{pmatrix}
			#1 & #2 \\
			#3 & #4
		\end{pmatrix}
	}
}
\newcommand{\stbt}[4]{
	\ensuremath{
		\begin{smatrix}
			#1 & #2 \\
			#3 & #4
		\end{smatrix}
	}
}

\begin{document}

\begin{abstract}
We use higher Coleman theory to construct a new $p$-adic $L$-function for $\GSp_4 \times \GL_2$. While previous works by the first author, Pilloni, Skinner and Zerbes had considered the $p$-adic variation of classes in the $H^2$ of Shimura varieties for $\GSp_4$, in this note we explore the interpolation of classes in the $H^1$, which detect critical values for a different range of weights, disjoint from the range covered by this earlier construction. Further, we show an interpolation property in terms of complex $L$-values using the algebraicity results established in previous work by the authors.
\end{abstract}

\renewcommand{\urladdrname}{\emph{ORCID}}

\author{David Loeffler}
\address[D.L.]{Mathematics Institute, University of Warwick, Coventry, UK}
\curraddr{UniDistance Suisse, Schinerstrasse 18, 3900 Brig, Switzerland}
\email{david.loeffler@unidistance.ch}
\urladdr{\href{https://orcid.org/0000-0001-9069-1877}{\upshape\path{0000-0001-9069-1877}}}

\author{\'Oscar Rivero}
\address{O.R.: Mathematical Sciences Research Institute / Simons Laufer Mathematical Sciences Institute, 17 Gauss Way, Berkeley, CA 94720, USA}
\curraddr{Department of Mathematics, Universidade de Santiago de Compostela, Spain}
\email{oscar.rivero@usc.es}

\thanks{Supported by ERC Consolidator grant ``Shimura varieties and the BSD conjecture'' (D.L.) and Royal Society Newton International Fellowship NIF\textbackslash R1\textbackslash 202208 (O.R.). This material is based upon work supported by the National Science Foundation under Grant No. DMS-1928930 while the authors were in residence at the Mathematical Sciences Research Institute in Berkeley, California, during the Spring 2023 semester.}


 \maketitle
 \setcounter{tocdepth}{1}
 \tableofcontents

\section{Introduction}

Let $\pi$ and $\sigma$ be cuspidal automorphic representations of $\GSp_4 / \QQ$ and $\GL_2 / \QQ$ respectively. Then we have a degree 8 $L$-function $L(\pi \times \sigma, s)$, associated to the tensor product of the natural degree 4 (spin) and degree 2 (standard) representations of the $L$-groups of $\GSp_4$ and $\GL_2$. If $\pi$ and $\sigma$ are algebraic, then this $L$-function is expected to correspond to a motive, and in particular we can ask whether it has critical values.

We suppose that $\pi$ (or, more precisely, its $L$-packet) corresponds to a holomorphic Siegel modular eigenform of weight $(k_1, k_2)$, for $k_1 \ge k_2 \ge 2$ integers, and that $\sigma$ corresponds to a holomorphic elliptic modular form of weight $\ell \ge 1$. For $L(\pi \times \sigma, s)$ to be a critical value, we must have $s = \tfrac{-(k_1 + k_2 + \ell- 4)}{2} + j$ for $j \in \ZZ$, so that $L(\pi \times \sigma, s) = L(V_p(\pi) \otimes V_p(\sigma), j)$ where $V_p(-)$ are the Galois representations corresponding to $\pi$ and $\sigma$; and the tuple $(k_1, k_2, \ell, j)$ has to satisfy one of 3 different sets of (mutually exclusive) inequalities, which we have outlined in more detail in the companion paper \cite{LR2}, corresponding to the cases $(A)$, $(D)$, $(F)$ in Table 1 of \emph{op.cit.}.
% $\GSp_4 \times \GL_2$\footnote{For the reader's convenience we clarify that $(k_1, k_2)$ of op.cit.~correspond to $(r_1 + 3, r_2 + 3)$ in the present notation, and the $c_1$ of \emph{op.cit.} is $\max(2s, 2-2s) \in \ZZ_{\ge 1}$ (while $c_2$ is the same as our $c_2$).}.
  In this paper, we focus on region $(D)$, which is given by the inequalities
\begin{equation}
 \label{eq:ineqs_j}
 \begin{aligned}
 k_1 - k_2 + 3 &\le \ell \le k_1 + k_2 - 3, \\
 \max(k_1, \ell) &\le j \le \min(k_2 + \ell + 3, k_1 + k_2 - 3).
 \end{aligned}
\end{equation}
The corresponding values of $s$ and $\ell$ are illustrated in the diagram below. (The ``off-centre'' regions $(B), (E)$, and the two grey diagonal lines, will be explained shortly.)

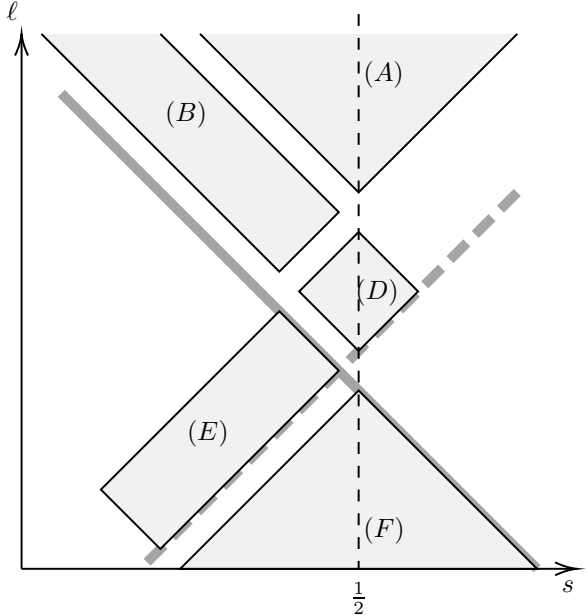
\tikzset{every picture/.style={line width=0.75pt}} %set default line width to 0.75pt

\begin{tikzpicture}[x=0.75pt,y=0.75pt,yscale=-1,xscale=1]
%uncomment if require: \path (0,347); %set diagram left start at 0, and has height of 347

%Straight Lines [id:da46752639414311403]
\draw [color={rgb, 255:red, 165; green, 165; blue, 165 }  ,draw opacity=1 ][line width=3.75]  [dash pattern={on 9pt off 4.5pt}]  (460,120) -- (270,310) ;
%Straight Lines [id:da7207496603743515]
\draw [color={rgb, 255:red, 165; green, 165; blue, 165 }  ,draw opacity=1 ][line width=3.75]    (230,70) -- (470,310) ;
%Straight Lines [id:da7919728211189669]
\draw    (210,310) -- (488,310) ;
\draw [shift={(490,310)}, rotate = 180] [color={rgb, 255:red, 0; green, 0; blue, 0 }  ][line width=0.75]    (10.93,-3.29) .. controls (6.95,-1.4) and (3.31,-0.3) .. (0,0) .. controls (3.31,0.3) and (6.95,1.4) .. (10.93,3.29)   ;
%Shape: Triangle [id:dp3275490883499559]
\draw  [fill={rgb, 255:red, 241; green, 241; blue, 241 }  ,fill opacity=1 ] (380,220) -- (470,310) -- (290,310) -- cycle ;
%Shape: Diamond [id:dp961347636667798]
\draw  [fill={rgb, 255:red, 241; green, 241; blue, 241 }  ,fill opacity=1 ] (380,140) -- (410,170) -- (380,200) -- (350,170) -- cycle ;
%Straight Lines [id:da2906033262064447]
\draw [fill={rgb, 255:red, 241; green, 241; blue, 241 }  ,fill opacity=1 ]   (460,40) -- (380,120) -- (300,40) ;
%Straight Lines [id:da6593245785109718]
\draw    (210,310) -- (210,42) ;
\draw [shift={(210,40)}, rotate = 90] [color={rgb, 255:red, 0; green, 0; blue, 0 }  ][line width=0.75]    (10.93,-3.29) .. controls (6.95,-1.4) and (3.31,-0.3) .. (0,0) .. controls (3.31,0.3) and (6.95,1.4) .. (10.93,3.29)   ;
%Straight Lines [id:da07020087697054223]
\draw  [dash pattern={on 4.5pt off 4.5pt}]  (380,30) -- (380,310) ;
%Shape: Polygon [id:ds8457648750180513]
\draw  [fill={rgb, 255:red, 241; green, 241; blue, 241 }  ,fill opacity=1 ] (340,180) -- (370,210) -- (280,300) -- (250,270) -- cycle ;
%Straight Lines [id:da8926622329233427]
\draw [fill={rgb, 255:red, 241; green, 241; blue, 241 }  ,fill opacity=1 ]   (280,40) -- (370,130) -- (340,160) -- (220,40) ;

% Text Node
\draw (201,22.4) node [anchor=north west][inner sep=0.75pt]    {$\ell $};
% Text Node
\draw (481,314.4) node [anchor=north west][inner sep=0.75pt]    {$s$};
% Text Node
\draw (381,52.4) node [anchor=north west][inner sep=0.75pt]    {$( A)$};
% Text Node
\draw (377,162.4) node [anchor=north west][inner sep=0.75pt]    {$( D)$};
% Text Node
\draw (381,282.4) node [anchor=north west][inner sep=0.75pt]    {$( F)$};
% Text Node
\draw (380,315.4) node [anchor=north] [inner sep=0.75pt]    {$\tfrac{1}{2}$};
% Text Node
\draw (292,233.4) node [anchor=north west][inner sep=0.75pt]    {$( E)$};
% Text Node
\draw (281,72.4) node [anchor=north west][inner sep=0.75pt]    {$( B)$};

\end{tikzpicture}

We shall now consider the case when $\pi$ and $\sigma$ vary through $p$-adic families. We consider Coleman families $\underline{\pi}$ for $\GSp_4$ (over some 2-dimensional affinoid space $U \subset \cW \times \cW$, where $\cW$ is the space of characters of $\ZZ_p^\times$), and similarly $\underline{\sigma}$ for $\GL_2$, over a 1-dimensional affinoid $U' \subset \cW$.

Following \cite{LZvista} and \cite{LR2}, we may conjecture that there exist 3 different $p$-adic $L$-functions in $\cO(U  \times U'\times \cW)$, denoted by $\mathcal{L}^{\spadesuit}(\underline{\pi} \times \underline{\sigma}, -)$ for $\spadesuit \in \{ (A), (D), (F)\}$, whose values at integer points $(k_1, k_2, \ell, j)$ satisfying the inequalities \eqref{eq:ineqs_j} interpolate the corresponding complex $L$-values. (These depend on various auxiliary data, which we suppress for now.)

In \cite{LZ21-erl}, building on the earlier work \cite{LPSZ1}, we proved a weakened form of this conjecture for region $(F)$: we constructed a $p$-adic $L$-function over a codimension-1 subspace of the parameter space $U  \times U'\times \cW$, interpolating $L$-values in region $(F)$ and lying at the ``right-hand edge'' of the critical strip. Thus, for each $(k_1, k_2, \ell)$ such that $\ell \le k_1 - k_2 + 1$, our $p$-adic $L$-function captures just one among the (possibly) many critical values of the $L$-function of the weight $(k_1, k_2, \ell)$ specialisation of $\underline{\pi} \times \underline{\sigma}$. This corresponds to the solid grey diagonal line in the above figure. We also showed that certain (non-critical) values of this $p$-adic $L$-function, corresponding to the elongation of the diagonal line to meet region $(E)$, were related to syntomic regulators of Euler system classes constructed in \cite{HJS20}; the region $(E)$ in the above diagram is precisely the range of weights in which the geometric Euler system classes of \emph{op.cit.} are defined.

\begin{note} We would also expect a second Euler system construction for weights in region $(B)$, but this is only conjectural at present.
\end{note}

The goal of this paper is to prove the analogue for region $(D)$ of the first main result proved for region $(F)$ in \cite{LZ21-erl}. That is, we define a $p$-adic $L$-function interpolating $L$-values along the ``lower right edge'' of region $(D)$, i.e.~for $(k_1, k_2, \ell, j)$ satisfying the conditions
\[ k_1 - k_2 + 3 \le \ell \le k_1, \qquad j = k_2 + \ell - 3 \quad \left(\Leftrightarrow\ s = \tfrac{\ell - k_1 + k_2 - 2}{2}\right).\]
So this $p$-adic $L$-function again lives over a codimension 1 subspace of the 4-dimensional parameter space, but a different one from that of \cite{LZ21-erl}: it is indicated by the dotted grey line in the figure. We conjecture, but do not prove here, a relation between this new $p$-adic $L$-function and syntomic regulators in region $(E)$; we hope to return to this in a subsequent work.

\begin{remark}
Both in the present paper and in \cite{LZ21-erl}, the reason why we lose one variable in the construction is that we do not know how to work with \emph{nearly-holomorphic} modular forms in the framework of higher Coleman theory. More precisely, $L$-values anywhere in region $(F)$, and in the ``lower half'' of region $(D)$, can be interpreted algebraically via cup-products in coherent cohomology; but the Eisenstein series appearing in these expressions are only holomorphic if $s$ lies at the upper or lower limit of the allowed range -- otherwise, they are nearly-holomorphic but not holomorphic. We are optimistic that future developments in higher Hida/Coleman theory may circumvent this barrier, allowing the construction of $p$-adic $L$-functions over the full 4-dimensional parameter space with interpolating properties in region $(F)$ or region $(D)$.
\end{remark}

\subsection*{The main result}

It is convenient to re-index the weights by setting $(r_1, r_2) = (k_1 - 3, k_2 - 3)$; for region $(D)$ to be non-empty we need $r_1 \ge r_2 \ge 0$, and in this case, $(r_1, r_2)$ is the highest weight of the algebraic representation of $\GSp_4$ for which $\pi$ is cohomological. Let $\chi_{\pi}$ (resp. $\chi_{\sigma}$) the central character of $\pi$ (resp. $\sigma$).

To define the (imprimitive) $p$-adic $L$-function, we need to consider the following objects. Here, $P$ denotes a point in $U$ and $Q$ a point in $U'$.
\begin{itemize}
\item A set of local conditions encoded in terms of the local data $\gamma_S$, introduced in \S\ref{sec:tame} and \S\ref{sec:ZS}, and appearing in the factor $Z_S(\pi_P \times \sigma_Q, \gamma_S)$.
\item A degree eight Euler factor $\mathcal E^{(D)}(\pi_P \times \sigma_Q)$, where $\pi_P$ (resp. $\sigma_Q$) stands for the specialization of $\underline{\pi}$ (resp. $\underline{\sigma}$) at the point $P$ (resp. $Q$). This is consistent with the predictions of \cite[\S4.3]{LZvista} and the fact that the Galois representation is 8-dimensional. The precise computation and description of the Euler factor is done in \cite[\S7]{LR2}.
\item The completed (complex) $L$-function $\Lambda(\pi_P \times \sigma_Q, s)$.
\item A basis $\underline{\xi} \otimes \underline{\eta}$ of the space $S^1(\underline{\pi}) \otimes S^1(\underline{\sigma})$, as introduced in \cite[Def. 10.4.1]{LZ21-erl}. The $p$-adic $L$-function does depend on that choice.
\item The complex (resp. $p$-adic) period $\Omega_{\infty}(\pi_P, \sigma_Q)$ (resp. $\Omega_p(\pi_P, \sigma_Q)$), introduced in Definition \ref{def:periods} and depending also on the specialization $\xi_P \otimes \eta_Q$ of the canonical differential $\underline{\xi} \otimes \underline{\eta}$ at $(P,Q)$.
\item The Gauss sum attached to $\chi_{\sigma}^{-1}$, denoted by $G(\chi_{\sigma}^{-1})$.
\end{itemize}

Further, we need to introduce the notion of {\it nice critical point}. We say a point $(P,Q)$ of $U \times U'$ is nice if $P = (r_1,r_2)$ and $Q = (\ell)$ are integer points, with $P$ nice for $\underline{\pi}$ and $Q$ nice for $\underline{\sigma}$, according to the definitions of Section \ref{sec:function}. Further, we say $(P,Q)$ is nice critical if we also have $r_1-r_2+3 \leq \ell \leq r_1+3$.

The main theorem we prove in this note, using in a crucial way the algebraicity result of \cite{LR2}, is the following.

\begin{theorem}
There exists a $p$-adic $L$-function $\mathcal L_{p,\gamma_S}^{\imp}(\underline{\pi} \times \underline{\sigma})$ satisfying the following interpolation property: if $(P,Q)$ is nice critical, then \[ \frac{\mathcal L_{p,\gamma_S}^{\imp}(\underline{\pi} \times \underline{\sigma})(P,Q)}{\Omega_p(\pi_P,\sigma_Q)} = Z_S(\pi_P \times \sigma_Q, \gamma_S) \cdot \mathcal E^{(D)}(\pi_P \times \sigma_Q) \cdot \frac{G(\chi_{\sigma}^{-1}) \Lambda(\pi_P \times \sigma_Q, \tfrac{\ell - k_1 + k_2 - 2}{2})}{\Omega_{\infty}(\pi_P, \sigma_Q)}, \]
where $(k_1, k_2, \ell)$ are such that $\pi_P$ has weight $(k_1, k_2)$ and $\sigma_Q$ has weight $\ell$.
\end{theorem}

The approach we follow to establish the theorem is the following:
\begin{enumerate}
\item Use results of Harris and Su to express the automorphic period to be computed as a cup product in the coherent cohomology of a Shimura variety associated with $\GL_2 \times \GL_2$. (This has already been carried out in \cite{LR2}.)
\item Use higher Coleman theory to reinterpret the cup product in terms of a pairing in coherent cohomology over certain strata in the adic Shimura varieties.
\item Use the families of automorphic forms $\underline{\pi}$ and $\underline{\sigma}$ in order to define the $p$-adic $L$-function $\mathcal L_{p,\gamma_S}^{\imp}(\underline{\pi} \times \underline{\sigma}; \underline{\xi})$.
\item Derive an interpolation formula at critical points using the compatibility of the cup-product with specialisation.
\end{enumerate}

\begin{remark}
 For this specific value $s =  \tfrac{\ell - k_1 + k_2 - 2}{2}$, we can write $L(\pi_P\times \sigma_Q, s) = L(V, 0)$, where $V$ is the Galois representation $V(\pi_P) \otimes V(\sigma_Q)(k_2 + \ell - 3)$. This Galois representation always has one of its Hodge--Tate weights equal to 0, which gives an intuitive explanation of why it should be ``easier'' to interpolate $L$-values along this subspace of the parameter space rather than over the entire 4-dimensional parameter space incorporating arbitrary cyclotomic twists.

 If we specialise at a fixed $P$, giving a one-variable $p$-adic $L$-function $L_{p,\gamma_S}^{\imp}(\pi \times \underline{\sigma})$ associated to a fixed $\pi$ and a $\GL_2$ family $\underline{\sigma}$, and we choose this $\underline{\sigma}$ to be a family of ordinary CM forms (arising from an imaginary quadratic field $K$ in which $p$ is split), then $L$-values interpolated by $\mathcal L_{p,\gamma_S}^{\imp}(\underline{\pi} \times \underline{\sigma})$ can be interpreted as values of the $L$-function of $\pi$ twisted by Gr\"ossencharacters of $K$; and the restriction on the value of $s$ implies that the Gr\"ossencharacters arising have infinity-types of the form $(n, 0)$. We expect that this $L$-function should have an interpretation as a ``$\fp$-adic $L$-function'' , interpolating twists by characters of the ray class group of $K$ modulo $\fp^\infty$, for a specific choice of prime $\fp$ above $p$; this will be pursued in more detail elsewhere.
\end{remark}

\subsection*{Connection with other works}

In \cite{LZvista}, the authors work in the setting of cusp forms for the larger group $\GSp_4 \times \GL_2 \times \GL_2$ and conjecture the existence of 6 different $p$-adic $L$-functions interpolating Gross--Prasad periods, corresponding to the `sign $+1$' regions $(a)$, $(a')$, $(c)$, $(d)$, $(d')$ and $(f)$ in the diagrams of \emph{op.cit.}. The case of region $(f)$ was covered by the work of Loeffler--Pilloni--Skinner--Zerbes \cite{LPSZ1} (see also \cite{LZ21-erl}) using higher Hida and Coleman theory, and the $p$-adic $L$-function for region $(c)$ was announced by Bertolini--Seveso--Venerucci, also using tools from coherent cohomology. Note however that the works cited only cover the case when one of the $\GL_2$-forms is an Eisenstein series.

If one formally replaces one of the two cusp forms by an Eisenstein series, then the Gross--Prasad period becomes Novodvorsky's integral computing the degree 8 $L$-function for $\GSp_4 \times \GL_2$; and regions $(a), (b), (d), (e), (f)$ correspond to the regions $(A)$, $(B)$, $(D)$, $(E)$, $(F)$ of the $\GSp_4 \times \GL_2$ figure above (while the arithmetic meaning of the remaining regions $(a'), (b'), (d'), (c)$ is less clear in this case). The methods we develop in the present work for region $(D)$ can be straightforwardly modified to interpolate $\GSp_4 \times \GL_2 \times \GL_2$ Gross--Prasad periods along one edge of region $(d)$ (and its mirror-image $(d')$).

For weights in the ``off-centre'' regions $(B)$ and $(E)$, the complex $L$-value $L(\pi \times \sigma, s)$ vanishes to order precisely 1, due to the shape of the archimedean $\Gamma$-factors. Beilinson's conjecture predicts the existence of canonical motivic cohomology classes whose complex regulators are related to $L'(\pi \times \sigma, s)$; and we expect the images of these classes in $p$-adic \'etale cohomology to form Euler systems. For weights in region (E), an Euler system has been obtained in recent work of Hsu, Jin and Sakamoto \cite{HJS20}; and Zerbes and the first author showed in \cite{LZ20b-regulator} that the syntomic regulators of these classes are related to values (outside its domain of interpolation) of the $p$-adic $L$-function interpolating critical values in region $(F)$. In the last section of this article, we discuss the kind of reciprocity law one can expect relating the cohomology classes of \cite{HJS20} with the $p$-adic $L$-function of this article. We hope to come back to this question in a forthcoming work.

Aside from the works listed above, the only other prior work we know of which treats $p$-adic interpolation of $\GSp_4 \times \GL_2$ $L$-values is the PhD thesis of M.~Agarwal \cite{agarwal}. Agarwal's construction gives a one-variable $p$-adic $L$-function, which appears to correspond to the restriction of our 3-variable function to the line where $k_1 = k_2 = \ell = k$ for a parameter $k$, although his methods are very different from ours (using an Eisenstein series on the unitary group $U(3, 3)$).

Since the release of this preprint, Z. Liu \cite{liu23} has constructed a new $p$-adic $L$-function in the case of $\GSp_4 \times \GL_2$, using the same integral representation of \cite{agarwal}; her construction includes the cyclotomic variable and proceeds in a different way than ours, using Klingen Eisenstein series on $\text{GU}(2,2)$ with Bessel models of representations of $\GSp_4$. Further, the preprint \cite{GPR}, which appeared also after the release of this preprint, develops new tools which are likely to allow the construction of a four variable $p$-adic $L$-function for regions $(D)$ and $(F)$.

One of the differences among the present paper and most of the works discussed above is that here we interpret Siegel modular forms in the $H^1$ of the Siegel modular variety, and not in the $H^2$ as in \cite{LPSZ1}; further, the $\GL_2$-form is naturally seen as an element in the $H^1$ of the modular curve. These choices are a direct reflection of the different weight regions considered in this setting.

\subsection*{Acknowledgements} The authors would like to thank Sarah Zerbes for informative conversations related to this work, and to Andrew Graham for some comments on an earlier version of the paper. They also express their gratitude to the anonymous referee, whose comments contributed to improve the presentation and to correct some inaccuracies.

%\newpage

\section{Setup: groups and Hecke parameters}

 \subsection{Groups}

  We denote by $G$ the group scheme $\GSp_4$ (over $\ZZ$), defined with respect to the anti-diagonal matrix $J = \begin{smatrix}&&&1\\ &&1 \\ &-1 \\ -1 \end{smatrix}$; and we let $\nu$ be the multiplier map $G \to \GG_m$. We define $H = \GL_2 \times_{\GL_1} \GL_2$, which we embed into $G$ via the embedding
  \[ \iota \, : \, \Big[
   \begin{pmatrix}a&b\\c&d\end{pmatrix}, \begin{pmatrix}a'&b'\\c'&d'\end{pmatrix} \Big]
   \mapsto \begin{pmatrix}a&&&b\\&a'&b'\\&c'&d'\\c&&&d\end{pmatrix}.
  \]
  We sometimes write $h_i$ for the $i$-th $\GL_2$ factor of $H$. We write $T$ for the diagonal torus of $G$, which is contained in $H$ and is a maximal torus in either $H$ or $G$.
 %  write $H_i$ for the $i$th $\GL_2$-factor of $H$, $T_i$ for the diagonal torus of $H_i$, and $\Delta$ for the common quotient identified with $\GG_m$ via the character $\nu$, so $T = T_1 \times_{\Delta} T_2$. For $i=1,2$, write $\varpi_i \, : \, T \rightarrow T_i$ for the natural projection map.

 \subsection{Parabolics}

  We write $B_G$ for the upper-triangular Borel subgroup of $G$, and $P_{\Si}$ and $P_{\Kl}$ for the standard Siegel and Klingen parabolics containing $B$, so
  \[ P_{\Si} =
   \begin{smatrix}\star & \star & \star & \star \\ \star & \star & \star & \star\\
   && \star & \star\\ && \star & \star \end{smatrix}, \qquad
   P_{\Kl} =
   \begin{smatrix}\star & \star & \star & \star \\  & \star & \star & \star\\
   &\star & \star & \star\\ &&& \star \end{smatrix}.
  \]
  We write $B_H = \iota^{-1}(B_G) = \iota^{-1}(P_{\Si})$ for the upper-triangular Borel of $H$.

  We have a Levi decomposition $P_{\Si} = M_{\Si} N_{\Si}$, with $M_{\Si} \cong \GL_2 \times \GL_1$, identified as a subgroup of $G$ via
  \[
   (A, u) \mapsto \begin{pmatrix}A\\&uA'\end{pmatrix},\qquad A' \coloneqq \begin{pmatrix}&1\\1\end{pmatrix}{}^t\!A^{-1} \begin{pmatrix}&1\\1\end{pmatrix}.
  \]
  In this paper $P_{\Si}$ and $M_{\Si}$ will be much more important than $P_{\Kl}$ and $M_{\Kl}$ (in contrast to \cite{LPSZ1}) so we shall often denote them simply by $P$ and $M$.
  The intersection $B_M \coloneqq M \cap B_G$ is the standard Borel of $M$; its Levi factor is $T$.

%\subsection{Twisted embeddings}

%\begin{definition}
%Let us write $\tau = \left(\tbt{1}{}{1}{1}, 1\right) \in M(\ZZ)$.
%\end{definition}

\subsection{Coefficient sheaves}

We retain the conventions about algebraic weights and roots of \cite{LZ21-erl}. In particular, we identify characters of $T$ with triples of integers $(r_1,r_2; c)$, with $r_1+r_2 = c$ modulo 2 corresponding to $\diag(st_1,st_2,st_2^{-1},st_1^{-1}) \mapsto t_1^{r_1} t_2^{r_2} s^c$. With our present choices of Borel subgroups, a weight $(r_1,r_2; c)$ is dominant for $H$ if $r_1,r_2 \geq 0$, dominant for $M_G$ if $r_1 \geq r_2$, and dominant for $G$ if both of these conditions hold. (We frequently omit the central character $c$ if it is not important in the context.)

For our further use, we briefly recall the conventions of loc.\,cit. about sheaves. The Weyl group acts on the group of characters $X^*(T)$ via $(w \cdot \lambda)(t) = \lambda(w^{-1}tw)$. As discussed in loc.\,cit., we can define explicitly $w_G^{\max}$, the longest element of the Weyl group, as well as $\rho = (2, 1; 0)$, which is half the sum of the positive roots for $G$. There is a functor from representations of $P_G$ to vector bundles on the compactified Siegel Shimura variety; and we let $\mathcal V_{\kappa}$, for $\kappa \in X^{\bullet}(T)$ that is $M_G$-dominant, be the image of the irreducible $M_G$-representation of highest weight $\kappa$. Given an integral weight $\nu \in X^\bullet(T)$ such that $\nu + \rho$ is dominant, we define \[ \kappa_i(\nu) = w_i(\nu+\rho) - \rho, \quad 0 \leq i \leq 3, \] here, as usual, $\rho$ is half the sum of the positive roots and the $w_i$ stand for the Kostant representatives of the Weyl group. These are the weights $\kappa$ such that representations of infinitesimal character $\nu^{\vee}+\rho$ contribute to $R\Gamma(S_K^{G,\tor},\mathcal V_{\kappa})$, where $\nu^{\vee}$ is the dual weight of $\nu$ and the superscript ``$\tor$'' stands for the toroidal compactification (for some choice of a projective cone decomposition). If $\nu$ is dominant (i.e. $r_1 \ge r_2 \ge 0$), they are the weights which appear in the \emph{dual BGG complex} computing de Rham cohomology with coefficients in the algebraic $G$-representation of highest weight $\nu$.

\subsection{Hecke parameters}

With the notations of the introduction, let $\pi$ be a cuspidal automorphic representation of $G$, and let $p$ be a prime. If $\pi_{\fin}$ is unramified at $p$, we write $\alpha, \, \beta, \, \gamma, \, \delta$ for the Hecke parameters of $\pi_p$, and $P_p(X)$ for the polynomial $(1-\alpha X) \ldots (1-\delta X)$. The Hecke parameters are algebraic integers over a number field $E$, and are well-defined up to the action of the Weyl group. If $\pi$ is non-CAP, which we shall assume\footnote{Equivalently (given our conditions on the weight), we require that $\pi$ is not a Saito--Kurokawa lift; for lifts of this type, the ratio of two among the Hecke parameters is $p$, so they cannot all have the same absolute value. Including Saito--Kurokawa lifts would add extra technical complications in our theory; and excluding them is no loss anyway, since if $\pi$ is such a lift, then $L(\pi \times \sigma, s)$ factors as a product of $L$-functions of $\GL_2$ and $\GL_2 \times \GL_2$, whose $p$-adic interpolation is well-understood.}, then the Hecke parameters all have complex absolute value $p^{w/2}$, where $w \coloneqq r_1+r_2+3$, and they satisfy $\alpha \delta = \beta \gamma = p^w \chi_{\pi}(p)$, where $\chi_{\pi}(p)$ is a root of unity.

%The fact that $\pi_p$ is unramified and generic implies that $(\pi_p')^{\Si(p)}$ is 4-dimensional, and the two operators \[ U_{p,\Si} = p^{-r_2} \cdot [\Si(p) \diag(p,p,1,1) \Si(p)], \quad U_{p,\Si}' = p^{-r_2} \cdot [\Si(p) \diag(1,1,p,p) \Si(p)] \] acting on this space both have eigenvalues $\{ \frac{\alpha \beta}{p^{r_2+1}}, \frac{\alpha \gamma}{p^{r_2+1}}, \frac{\beta \delta}{p^{r_2+1}}, \frac{\gamma \delta}{p^{r_2+1}} \}$. We can similarly define Klingen Hecke operators (see \cite[\S5.5]{LPSZ1} for a more extensive discussion).

Let $\Iw_G(p)$ denote the Iwahori subgroup. We shall consider the following operators in the Hecke algebra of level $\Iw_G(p)$, acting on the cohomology of any of the sheaves introduced before:
\begin{itemize}
\item The Siegel operator $\mathcal U_{\Si} = [\diag(p,p,1,1)]$, as well as its dual $\mathcal U_{\Si}' = [\diag(1,1,p,p)]$.
\item The Klingen operator $\mathcal U_{\Kl} = p^{-r_2} \cdot [\diag(p^2,p,p,1)]$, as well as its dual $\mathcal U_{\Kl}' = p^{-r_2} \cdot [\diag(1,p,p,p^2)]$.
\item The Borel operator $\mathcal U_B = \mathcal U_{\Si} \cdot \mathcal U_{\Kl}$, as well as its dual $\mathcal U_B' = \mathcal U_{\Si}' \cdot \mathcal U_{\Kl}'$.
\end{itemize}

%\subsection{Flag varieties and Bruhat cells}\label{sec:flag}

  %There are four orbits for the Borel $B_G$ acting on $\FL_G$, the \emph{Bruhat cells}, represented by a subset of the Weyl group of $G$, the \emph{Kostant representatives}, which are the smallest-length representatives of the quotient $W_M \backslash W_G$. We denote these by $w_0, \dots, w_3$; see \cite{LZ21-erl} for explicit matrices. Note that the cell $C_{w_i} = P \backslash P w_i B_G \subset \FL_G$ has dimension $\ell(w_i) = i$.

\section{Flag varieties and orbits}

The key technical input into our interpolation results is a detailed study of certain loci in flag varieties for $G$ and $H$, making explicit the theory of \cite{boxerpilloni20} for the groups $G$ and $H$, and studying how it interacts with restriction from $G$ to $H$.

We use the usual Roman letters $X, Y, \dots$ for algebraic varieties or schemes, while calligraphic letters $\cX, \cY, \dots$ or typewriter letters ${\tt X}, {\tt Y}, \dots$ denote adic spaces.

\subsection{Kostant representatives}

We write $\FL_G$ for the Siegel flag variety $P \backslash G$, with its natural right $G$-action. There are four orbits for the Borel $B_G$ acting on $\FL_G$, represented by a subset of the Weyl group of $G$, the \emph{Kostant representatives} (a distinguished set of representatives for the quotient $W_{M_G} \backslash W_G$ where $M_G$ is the Levi of $P$). We denote these by $w_0, \dots, w_3$; see \cite{LZ21-erl} for explicit matrices.

\begin{definition}
 As in \cite[\S 3.1]{boxerpilloni20}, we write $C_{w_i}^G$ for the orbit $P \backslash P w_i B_G$, a locally closed subvariety of $\FL_G$ of dimension $i$.
\end{definition}

\begin{remark}
For $g \in G$, we can determine which cell $C_{w_i}^G$ contains the point $Pg \in \FL_G$ via a criterion in terms of the span of the rows of the bottom left $2\times 2$ submatrix of $g$, as in Remark 5.1.2 of \cite{LZ21-erl}.
\end{remark}

\begin{remark}
Note that $C_{w_0}^G$ and $C_{w_3}^G$ are stable under $P$, while $C_{w_1}^G \sqcup C_{w_2}^G$ forms a single $P$-orbit.
\end{remark}

Analogously, for the $H$-flag variety $\FL_H = B_H \backslash H$, we have 4 Kostant representatives $w_{00} = \mathrm{id}$, $w_{10} = \left(\begin{pmatrix}0&1\\-1&0\end{pmatrix}, \mathrm{id}\right)$, similarly $w_{01}$, $w_{11}$ (with the cell $C_{w_{ij}}^H$ having dimension $i + j$). (This is the whole of the Weyl group of $H$, since the Levi subgroup of $M_H = T$ is trivial.)

\begin{remark}
Either for $G$ or for $H$, each cell will determine a subspace of the Iwahori-level Shimura variety (as an adic space), via pullback along the Hodge--Tate period map. This is the locus where the relative position of the Hodge filtration and level structure on the $p$-divisible group lies in the given Bruhat cell. In particular, the ``smallest'' cell ($w_0$ or $w_{00}$) corresponds to the multiplicative locus, and the ``largest'' one to the \'etale locus.
\end{remark}

\subsection{A twisted embedding of flag varieties}\label{sec:twisted}

 Consider the elements
 \begin{align*}
  \tau &= \begin{smatrix}1 \\1&1\\&&1\\&&-1&1\end{smatrix} \in M(\Zp),
  &
 \tau^{\sharp} &= \iota(w_{01})^{-1} \tau w_2 \in G(\Zp).
 \end{align*}

 Note that $\tau$ was denoted $\gamma$ in \cite{LPSZ1}, but $\gamma$ was also used for a Satake parameter, so we use a different letter here. The element $\tau$ represents the unique open $T$-orbit for the $M$-flag variety $B_M \backslash M$.

 We will consider the translated embedding $\iota^\sharp: H \to G$ given by $h \mapsto \iota(h) \tau^\sharp$. The map $\FL_H \to \FL_G$ induced by $\iota^{\sharp}$ by construction sends $[w_{01}]$ to $[w_2]$. We also have projection maps $\pi_i: \FL_H \to \FL_{\GL_2} \cong \mathbf{P}^1$, and the product $(\iota^\sharp, \pi_1, \pi_2)$ evidently sends $w_{01}$ to $([w_2], [\mathrm{id}], [w])$ (where the unlabelled $w$ is the $\GL_2$ long Weyl element).

 If we equate $(x : y) \in \mathbf{P}^1$ with the orbit $B g \in B \backslash G$, where $g$ is any invertible matrix of the form $\stbt{\star}{\star}{x}{y}$, then $\iota^\sharp$ sends $\left((x: y), (X : Y)\right)$ to
 \[ P_{\Si} \cdot \begin{pmatrix}
 \star & \dots & \dots & \star\\
 \star & \dots & \dots & \star\\
 -X & Y & \star  & \star\\
 -y & x & \star & \star \end{pmatrix}.
 \]
 Using this and the explicit description of the Bruhat cells in terms of the bottom left corner of the matrix, we see that:
 \begin{itemize}
 \item the preimage of $C^G_{w_0}$ is empty;
 \item the preimage of $C^G_{w_1}$ is the point $((1:0), (0:1))$ (the image of $[w_{10}] \in \FL_H$);
 \item the preimage of $C^G_{w_2}$ is a copy of the affine line, corresponding to points of the form
 \[ B_H \left(\stbt 1 {} x 1, \stbt 1 {}{-x} 1\right) w_{01}.\]
 \end{itemize}

 \begin{proof}
  The condition for the above matrix to lie in $X^G_{w_2}$ is that $\stbt{-X}{Y}{-y}{x}$ be singular, i.e.~$Xx = Yy$; and the condition for it to lie in $C^G_{w_2}$ is that the span of the rows not be $(0:1)$, so $X \ne 0$ and $y \ne 0$. So we can wlog take $X = 1$ and $y = 1$, leaving the equation $Y = x$; i.e. our point was $B_H \left( \stbt{\star}{\star}{x}{1}, \stbt\star\star 1 x \right) = B_H \left( \stbt{\star}{\star}{x}{1}, \stbt\star\star {-x}{1}\right) w_{01}$.
 \end{proof}

 \begin{notation}
  We write $X^G_{w} = \bigcup_{w' \le w} C^G_{w'}$ (closed subvariety), and $Y^G_w = \bigcup_{w' \ge w} C^G_{w'}$ (open subvariety).
 \end{notation}

\begin{proposition}
 We have
 \[
  (\iota^\sharp)^{-1}\left( X^G_{w_2}\right) \cap \pi_2^{-1}\left( Y^{\GL_2}_{w} \right) =  (\iota^\sharp)^{-1}\left( C^G_{w_2}\right) \cap \pi_1^{-1}(C^{\GL_2}_w \cdot w^{-1}) \cap \pi_2^{-1}\left( C_{w}^{\GL_2} \right).
 \]
\end{proposition}

(Note that the translate $C^{\GL_2}_w \cdot w^{-1}$ is the ``big cell at the origin'', $B \backslash B \overline{B}$. )

\begin{proof}
 Since the single point $(\iota^{\sharp})^{-1}(C^G_{w_0} \cup C^G_{w_1}) = [w_{10}]$ does not map to $Y_w^{\GL_2}$ under $\pi_2$, we conclude that $(\iota^\sharp)^{-1}\left( X^G_{w_2}\right) \cap \pi_2^{-1}\left( Y^{\GL_2}_{w} \right)$ is equal to $(\iota^\sharp)^{-1}(C^G_{w_2})$. We saw above that this subvariety is a copy of the affine line, and its image under the $\pi_i$ is as stated.
\end{proof}

\subsection{Some tubes}

Let ${\tt FL}_G$ denote the analytification of $\FL_G$, as an adic space over $\Qp$, and similarly for $H$ and for $G \times H$, so ${\tt FL}_{(G \times H)} = {\tt FL}_G \times\mathbf{P}^{1, \an} \times \mathbf{P}^{1, \an}$.

We will now define loci inside these spaces, using the tubes of various subvarieties of the special fibres. As usual $\cX^G_w$ denotes the tube of $X^G_{w, \FF_p}$ in ${\tt FL}_G$, etc. We shall set
 \[ \tZ_0 = \overline{\cX^G_{w_2}} \times \mathbf{P}^{1, \an} \times \overline{\cY^{\GL_2}_w}, \]
 and
 \[ \tU_0 = \cY^G_{w_2} \times \cX_{\mathrm{id}}^{\GL_2} \times \mathbf{P}^{1, \an}.\]
 Then $\tZ_0$ is closed, $\tU_0$ is open, and both are stable under the action of $\Iw_G \times \Iw_{\GL_2} \times \Iw_{\GL_2}$; and $\tU_0 \cap \tZ_0$ is a partial closure of the $(w_2, \mathrm{id}, w)$ Bruhat cell for $G \times \GL_2 \times \GL_2$.

 We need to allow smaller ``overconvergence radii'', for which we use the action of the element $\eta_G = \diag(p^3, p^2, p, 1)$ and its cousin $\eta = \stbt{p}{}{}{1}$.

 \begin{definition}
  Let us set
  \( \tZ_m = \tZ_m^G \times \tZ_m^H, \)
  where
  \[ \tZ_m^G = \overline{\cX^G_{w_2}} \cdot \eta_G^m, \qquad\tZ_m^H = \mathbf{P}^{1, \an} \times \left(\overline{\cY^{\GL_2}_w} \cdot \eta^{-m} \stbt{1}{\Zp}{0}{1}\right).\]
 \end{definition}

 We have $\tZ_0 \supseteq \tZ_1 \supseteq \tZ_2 \dots$, by \cite[Lemma 3.4.1]{boxerpilloni20}, and $\tZ_m$ is stable under $\Iw(p^t)$ for $t \ge 3m + 1$.

 On the other hand, we can define $\tU_n = \tU_n^G \times \tU_n^H$, where
 \[ \tU_n^G = \cY^{G}_{w_2} \cdot \eta_G^{-m} N_{B_G}(\Zp), \qquad \tU_n^H = \left( \cX^{\GL_2}_{\mathrm{id}} \cdot \eta^n \right) \times \mathbf{P}^{1, \an}.\]
 Again, we have $\tU_0 \supseteq \tU_1 \supseteq \dots$, and $\tU_n$ is stable under $\Iw(p^t)$ for $t \ge n+1$.

\subsection{Explicit coordinates}

 For $m \in \QQ$, we define the subsets of the adic projective line given by \[ \mathcal B_m = \{ |.| \, : \, |z| \leq p^m \}, \quad \bar{\mathcal B}_m = \cap_{m'<m} \mathcal B_{m'}, \quad \mathcal B_m^{\circ} = \cup_{m'>m} \mathcal B_{m'}, \quad \bar{\mathcal B}_m^{\circ} = \{ |.| \, : \, |z|<|p|^m \}, \] that satisfy the inclusions $\mathcal B_m^{\circ} \subset \bar{\mathcal B}_m^{\circ} \subset \mathcal B_m \subset \bar{\mathcal B}_m$.

 We can identify the Zariski-open neighbourhood $U_{w_2} = P \backslash P \overline{P} w_2$ of $[w_2] \in \FL_G$ with $\AA^3$, via the map $P \backslash P\begin{smatrix}1\\ &1 \\  x & y & 1\\ z & x & & 1\end{smatrix} w_2$. Then one computes that
 \[ \overline{\cX^G_{w_2}} \cap U_{w_2}^{\mathrm{an}} =
  \left\{ (x, y, z): x \notin \cB_0 \text{ or } y \notin \cB_0 \text{ or } z \in \overline{\cB}_0^\circ\right\},\]
 and $\eta_G$ preserves $U_{w_2}^{\an}$ and acts in these coordinates via $(x, y, z) \mapsto (p^{-1} x, p^{-3}y, pz)$. Thus
 \[ \tZ_m^G \cap U_{w_2}^{\mathrm{an}} =
 \left\{ (x, y, z): x \notin \cB_{-m} \text{ or } y \notin \cB_{-3m} \text{ or } z \in \overline{\cB}_m^\circ\right\},\]
 and a similar computation identifies $\overline{\cY^{\GL_2}_w}$ with $\overline{\cB}_0$, and $\overline{\cY^{\GL_2}_w} \cdot \eta^{-m}\stbt{1}{\Zp}{0}{1}$ with $\overline{\cB}_m + \Zp$.

we can compute $\tU_n^G$ in coordinates as
 \[ \tU_n^G \cap U_{w_2}^{\an} = \left\{ (x, y, z): x \in \cB_n + \Zp, y \in \cB_{3n} + \Zp\right\},\]
 with no condition on $z$; and the projection to the first $\GL_2$ coordinate is just $\cB^\circ_n$.

 \begin{lemma}
  The intersection $\tZ^G_m \cap \tU_n^G$ is contained in $U^{\an}_{w_2}$, for all $m, n \ge 0$.
 \end{lemma}

 \begin{proof}
  It suffices to check this for $(m, n) = (0, 0)$; see Lemma 3.3.21 of \cite{boxerpilloni20}.
 \end{proof}

\subsection{Pullback to $H$}

 Guided by the zeta-integral computations of \cite{LR2}, we shall consider the map
 \[ \iota^{\sharp\sharp}: \FL_H \to \FL_G \times \FL_H,\qquad
 h \mapsto \left(\iota^\sharp(h), h_1 \stbt{p^t}{}{}1, h_2\right). \]
 for some $t \ge 1$.

 \begin{proposition}
  If $m > 3n \ge 0$, then
  \[ (\iota^{\sharp\sharp})^{-1}(\tZ_m \cap \tU_n) = (\iota^{\sharp\sharp})^{-1}(\tZ_m), \]
  and in particular this preimage is closed in $FL_H$.
 \end{proposition}

 \begin{proof}
  We know that the pullback of $\tZ_0$ is contained in the big cell, so we can compute it in coordinates. We find that the inequalities on $(z_1, z_2)$ for it to land in $\tZ_m$ are:
  \[ z_1 + z_2 \in \overline{\cB_m^\circ}, \qquad z_2 \in \overline{\cB_m} + \Zp.\]
  For $\tZ_m \cap \tU_n$ we add the extra inequalities
  \[ z_2 \in \cB_{3n} + \Zp, \qquad p^t z_1 \in \cB^\circ_n.\]
  If $m > 3n$ then the latter equations are a consequence of the former.
 \end{proof}

\subsection{Period maps and overconvergent cohomology}

We consider the analytifications $\mathcal S_{G,K} = (S_K \times \Spec(\QQ_p))^{\an}$ and $\mathcal S_{G,K}^{\tor} = (S_{G,K}^{\tor} \times \Spec(\QQ_p))^{\an}$, and similarly for $H$ and $G \times H$ (denoted always by calligraphic letters). Write $\mathcal S_{G,K^p}^{\tor}$ for the perfectoid space $\varprojlim_{K_p} \mathcal S_{G,K^pK_p}^{\tor}$, which allows us to consider the Hodge--Tate period map \[ \pi_{\HT,G}^{\tor} : \mathcal S_{G,K^p}^{\tor} \longrightarrow {\tt FL}_G, \] which for every open compact $K_p \subset G(\QQ_p)$ descends to a map of topological spaces (c.f. \cite[\S4.5]{boxerpilloni20} \[ \pi_{\HT,G,K_p}^{\tor} : \mathcal S_{G,K^pK_p}^{\tor} \longrightarrow {\tt FL}_G/K_p. \] There are also analogous maps for $H$ and $G \times H$, for which we use the same notations.

Then \cite[Thm. 6.2.1]{LZ21-erl} about the compatibility properties of the period map holds in the same way, replacing the level structure $K_{\diamond}^H$ by $K_{\triangle}^H$, defined as follows.

\begin{definition}
Let $K_{\triangle}^H(p^t) = K_{\Iw}^H(p^t) \cap \tau^{\sharp} K_{\Iw}^G(p^t) (\tau^{\sharp})^{-1}$. It is concretely given by \[ K_{\triangle}^H(p^t) = \{ h \in H(\ZZ_p) \mid h = \Big( \begin{pmatrix}x&0\\0&z\end{pmatrix}, \begin{pmatrix}z&0\\0&x\end{pmatrix} \Big) \mod p^t \text{ for some } x,\,z \}. \]
\end{definition}

We just put $\mathcal S_{H,\triangle}^{\tor}$ for $\mathcal S_{G,K}^{\tor}$ to represent the choice of $K_{\triangle}^H(p^t)$.

\begin{proposition}
There is a commutative diagram of Hodge--Tate period maps \[
\xymatrix{
    \mathcal S_{H,\Iw}^{\tor}(p^t) \ar[r]^-{\pi_{\Iw}^H}
    & {\tt FL}_H/K_{H,\Iw}(p^t) \\
    \mathcal S_{H,\triangle}^{\tor}(p^t)
    \ar[u]^{\pr_{\triangle}}
    \ar[d]^{\iota^\sharp}
    \ar[r]^-{\pi_{\triangle}^H}
    & {\tt FL}_H/K_{\triangle}^H(p^t)
    \ar[u]^{\pr_{\triangle}}
    \ar[d]^{\iota^\sharp} \\
    \mathcal S_{G,\Iw}^{\tor}(p^t) \ar[r]^-{\pi_{\Iw}^G}
    & {\tt FL}_G/K_{G,\Iw}(p^t)
}
\]

\end{proposition}

The choices of neighbourhoods we have made are sufficient to get the maps working, including the compatibility with the classical cohomology via compact-support cohomology of $\tZ_m$.
%Write $\tZ_m^H$ for the corresponding preimage in $H$.
Let $\mathcal U_n^G \subset \mathcal S_{G \times H,\Iw}(p^t)$ and $\mathcal Z_m^H \subset \mathcal U_n^H \subset \mathcal S_{H,\triangle}(p^t)$ denote the preimages of the subsets ${\tt U}_n^G \subset {\tt FL}^G$ and ${\tt Z}_m^H \subset {\tt U}_n^H \subset {\tt FL}^H$ under the Hodge--Tate period maps $\pi_{\HT, G \times H, \Iw}^{\tor}$ and $\pi_{\HT,H,\triangle}^{\tor}$. Then we have the following commutative diagram.

\[ \xymatrix{
		R\Gamma(\mathcal U_0^G, \mathcal V) \ar[r] & R\Gamma(\mathcal S_{H,\triangle}^{\tor}(p^t), (\iota^{\sharp})^*(\mathcal V) \\
		R\Gamma_{\mathcal Z_m}(\mathcal U_0^G, \mathcal V) \ar[u] \ar[d] \ar[r] & R\Gamma_{\mathcal Z_m^H}(\mathcal S_{H,\triangle}^{\tor}(p^t), (\iota^{\sharp})^*(\mathcal V)) \ar[u] \\
        R\Gamma_{\mathcal Z_m}(\mathcal U_n^G, \mathcal V). \ar[ur] &
	} \]

Here, the horizontal maps correspond to $(\iota^{\sharp})^*$, while the vertical ones are the usual restriction and corestriction maps.

%\newpage

\section{Branching laws and sheaves of distributions}

In this section, we introduce the necessary tools to $p$-adically interpolate the automorphic vector bundles associated to representations of the Levi subgroups $M_G$ and $M_H$, which are the coefficient systems for the cohomology we study. We keep the notations of \cite[\S8]{LZ21-erl} and review some of the more relevant results of loc.\,cit., focusing on the changes we need in our setting. Note in particular that the discussions and results of \S6 hold verbatim, with the obvious changes in Prop. 6.4.1.

Along this section, we will frequently consider the projections of the embedding $\iota^{\sharp \sharp}$ on each factor: the first component corresponds to $\iota^{\sharp} : \FL_H \to \FL_G$, and the second, referred as $\iota_p$, is the map \[ \iota_p : \FL_H \to \FL_H, \qquad
(h_1,h_2) \mapsto \left(h_1 \stbt{p^t}{}{}1, h_2\right). \] We also write $\upsilon = \upsilon(p^t) = \left(\tbt{p^t}{}{}{1}, 1\right) \in M(\ZZ_p)$.

\subsection{Torsors}

We begin this section recalling a general procedure for Tate-twisting pro\'etale torsors, referring the reader to \cite[\S4.2]{graham21} for a more extensive discussion on the main properties of this operation.
Let $L/\mathbb Q_p$ be a finite extension and $\mathcal X/L$ a smooth adic space. Let $\mathcal T^{\times} \rightarrow \mathcal X$ denote the pro\'etale $\mathbb Z_p^{\times}$-torsor parametrising isomorphisms (of pro\'etale sheaves) $\mathbb Z_p \xrightarrow{\sim} \mathbb Z_p(1)$. The action of $\mathbb Z_p^{\times}$ is given as follows: for $\lambda \in \mathbb Z_p^{\times}$ and $\phi : \mathbb Z_p \xrightarrow{\sim} \mathbb Z_p(1)$, we set \[ \phi \cdot \lambda = \phi(\lambda \cdot -). \] Let $M$ be a smooth adic group over $\Spa L$ and suppose that we have a homomorphism $\mu : \mathbb Z_p^{\times} \rightarrow M$ whose image is contained in the centre of $M$.

\begin{definition}\label{def:twisting}
Let $\mathcal M \rightarrow \mathcal X$ be a (right) pro\'etale $M$-torsor. We define the twist of $\mathcal M$ along $\mu$ to be \[ {}^{\mu} \mathcal M := \mathcal M \times^{[\mathbb Z_p^{\times},\mu]} \mathcal T^{\times}, \] where the right-hand side is the quotient of $\mathcal M \times_{\mathcal X} \mathcal T^{\times}$ by the equivalence relation \[ (m \cdot \mu(\lambda),\phi) \sim (m,\phi \cdot \lambda^{-1}), \quad \text{ for all } m \in \mathcal M, \, \phi \in \mathcal T^{\times}, \, \lambda \in \mathbb Z_p^{\times}. \] This defines a pro\'etale $M$-torsor ${}^{\mu} \mathcal M \rightarrow \mathcal X$ via the action $(m,\phi) \cdot n = (m \cdot n, \phi)$, for $m \in \mathcal M$, $\phi \in \mathcal T^{\times}$ and $n \in M$.
\end{definition}

The map $x \mapsto x^{-1} \, : \, G \rightarrow \FL_G$ allows us to regard $G$ as a right $P_G$-torsor over $\FL_G$, and similarly to regard $G/N_G \rightarrow \FL_G$ as a right $M_G$-torsor. We consider their analytifications \[ {\tt P}^G \, : \mathcal G \rightarrow {\tt FL}_G \quad \text{ and } \quad {\tt M}^G \, : \, \mathcal G/\mathcal N_G \rightarrow {\tt FL}_G \] which are torsors over ${\tt FL}_G$ under the analytic groups $\mathcal P_G$ and $\mathcal M_G$ respectively. We similarly define torsors over the flag varieties $H$, $H_1$ and $H_2$.

\begin{definition}
Define $\mathcal P_{\HT}^G$ and $\mathcal M_{\HT}^G$ to be the pullbacks via $\pi_{\HT}^G$ of the torsors ${\tt P}^G$ and ${\tt M}^G$; there are right torsors over $\mathcal S_{G,\Iw}(p^t)$ for the groups $\mathcal P_G$ and $\mathcal M_G$. We similarly define $\mathcal P_{\HT}^H$ and $\mathcal M_{\HT}^H$, $\mathcal P_{\HT}^{H_i}$ and $\mathcal M_{\HT}^{H_i}$ for $i=1,2$.

Using Definition \ref{def:twisting} we can define ${}^{\mu}\mathcal P_{\HT}^G$, ${}^{\mu}\mathcal M_{\HT}^G$ and the analogous twisted objects for the torsors corresponding to $H$ and $H_i$.
\end{definition}

\begin{definition}
For $n>0$, let $\mathcal M_{G,n}^1$ be the group of elements which reduce to the identity modulo $p^n$. Define \[ \mathcal M_{G,n}^{\square} = \mathcal M_{G,n}^1 \cdot B_{M_G}(\ZZ_p), \] which is an affinoid analytic subgroup containing $\Iw_{M_G}(p^n)$. A similar definition applies to $M_H = T$; we write the group as $\mathcal T_n^{\square} = T(\ZZ_p) \mathcal T_n^1$.
\end{definition}

Consider in the same way \[ \mathcal T_n^{\diamond} = \{ \diag(t_1,t_2,\nu t_2^{-1}, \nu t_1^{-1}) \in \mathcal T_n^{\square} \, : \, t_1-t_2 \in \mathcal B_n \}. \]

As in \cite[Prop. 7.2.5]{LZ21-erl}, we also consider the \'etale torsors ${}^{\mu}\mathcal M_{\HT,n}^G$, ${}^{\mu}\mathcal M_{\HT,n,\Iw}^G$ and ${}^{\mu}\mathcal M_{\HT,n,\diamond}^H$ arising as the reduction of structure of the torsors ${}^{\mu}\mathcal M_{\HT}^G$ over $\mathcal U_n^G$, ${}^{\mu}\mathcal M_{\HT}^H$  over $\mathcal U_{\Iw,n}^H$ and ${}^{\mu}\mathcal M_{\HT}^H$  over $\mathcal U_n^H$, respectively. The following proposition is the key statement allowing for $p$-adic variation, and it mainly follows from the theory developed in \cite[\S4.6]{boxerpilloni20}.

\begin{proposition}
We have an equality of $\mathcal M_{G,n}^{\square}$-torsors over $\mathcal U_{n,\diamond}^H$: \[ (\iota^{\sharp})^*({}^{\mu}\mathcal M_{\HT,n,\Iw}^G) = {}^{\mu} \mathcal M_{\HT,n,\diamond}^H \times^{[\mathcal T_n^{\diamond},\tau]} \mathcal M_{G,n}^{\square}, \] where we regard $\mathcal T_n^{\diamond}$ as a subgroup of $\Iw_{M_G}(p^t) \mathcal M_{G,n}^1$ via conjugation by $\tau$.
\end{proposition}

\begin{proof}
This follows by checking the analogous statement on the flag variety, noting that there is a commutative diagram of adic space:
\[ \xymatrix{
		K_{\triangle}^H(p^t) \mathcal H_n^1 \ar[r]\ar[d] & K_{\Iw}^G(p^t) \mathcal G_n^1 \ar[d] \\
		\mathcal B^H \backslash \mathcal B^H w_{01} K_{\triangle}^H(p^t) \mathcal H_n^1 \ar[r] & \mathcal P^G \backslash \mathcal P^G w_2 K_{\Iw}^G(p^t) \mathcal G_n^1.
	} \]
Here, the vertical maps are given by $h \mapsto \mathcal B^H \backslash \mathcal B^H h^{-1}$ on the left, and $g \mapsto \mathcal P^G \backslash \mathcal P^G w_2 g^{-1}$ on the right; the lower horizontal map is $\iota^{\sharp}$ is $\mathcal B^H h \mapsto \mathcal P^G h \tau w_2$, and the map along the top making the diagram commute is $h \mapsto (\tau^{\sharp})^{-1} h \tau^{\sharp}$.

Then we may conclude as in \cite[Prop. 7.2.7]{LZ21-erl}.
\end{proof}

%\DLnote{This needs careful re-examination with our new setup, and it might end up looking rather messy, since the powers-of-$p$ diagonal element has to be taken care of.}

A straightforward adaptation of these techniques can be applied to the second factor $\iota_p$, yielding to an equality of $\mathcal M_{H,n}^{\square}$-torsors over $\mathcal U_{n,\diamond}^H$, \[ (\iota_p)^*({}^{\mu}\mathcal M_{\HT,n,\Iw}^H) = {}^{\mu}\mathcal M_{\HT,n,\diamond}^H \times^{[\mathcal T_n^{\diamond},\upsilon]} \mathcal M_{H,n}^{\square}, \] where we regard $\mathcal T_n^{\diamond}$ as a subgroup of $\Iw_{M_H}(p^t) \mathcal M_{H,n}^1$ via conjugation by $\upsilon$. Observe that conjugation by $\upsilon$ does not introduce denominators in any element of $M_H$, and hence the previous objects are well defined.

\subsection{Analytic characters and analytic inductions}

\begin{definition}
Let $n \in \QQ_{>0}$. We say that a continuous character $\kappa \, : \, \ZZ_p^{\times} \rightarrow A^{\times}$, for $(A,A^+)$ a complete Tate algebra, is $n$-analytic if it extends to an analytic $A$-valued function on the affinoid adic space \[ \ZZ_p^{\times} \cdot \mathcal B_n \subset \GG_m^{\ad}. \] This definition extends to characters $T(\ZZ_p) \rightarrow A^{\times}$: the $n$-analytic characters are exactly those which extend to $\mathcal T_n^{\square}$.
\end{definition}

Let $n_0>0$ and assume that $\kappa_A \, : \, T(\ZZ_p) \rightarrow A^{\times}$ is an $n_0$-analytic character. For $? \in \{G,H\}$ and $n \geq n_0$, let $\mathcal M_{?,n}^1$ be the affinoid subgroup of $\mathcal M_?$ defined above, and let $B_{M_G}$ be the Borel of $M_?$.

\begin{definition}
For $n \geq n_0$, define \[ \begin{aligned} V_{G,\kappa_A}^{n-\an} & = \anInd_{(\mathcal M_{G,n}^{\square} \cap \mathcal B_G)}^{(\mathcal M_{G,n}^{\square})} (w_{0,M_?} \kappa_A) \\ & = \Big\{ f \in \mathcal O(\mathcal M_{G,n}^{\square}) \hat \otimes A \, : \, f(mb) = (w_{0,M} \kappa_A)(b^{-1}) f(m), \, \forall m \in \mathcal M_{G,n}^{\square}, \, \forall b \in \mathcal M_{G,n}^{\square} \cap \mathcal B_G \Big \}. \end{aligned} \]

We define a left action of $\mathcal M_{G,n}^{\square}$ on $V_{G,\kappa_A}^{n-\an}$ by $(h \cdot f)(m) = f(h^{-1}m)$.

Write $D_{G,\kappa_A}^{n-\an}$ for the dual space, and $\langle \cdot,\cdot \rangle$ for the pairing between these; we equip $D_{G,\kappa_A}^{n-\an}$ with a left action of the same group $\mathcal M_{G,n}^{\square}$, in such a way that $\langle h \mu, hf \rangle = \langle \mu ,f \rangle$.
\end{definition}

%As shown in \cite[Prop. 8.2.2]{LZ21-erl}, for a character $\kappa$ of the form $(\rho_1,\rho_2;\omega)$ the action of $(\begin{pmatrix}a&b\\c&d\end{pmatrix}, \nu)$ on $f \in \mathcal O(\mathcal B_n) \hat \otimes A$ is given by \[ ((\begin{pmatrix}a&b\\c&d\end{pmatrix},\nu)f)(z) = f \Big( \frac{az-c}{-bz+d} \Big) (-bz+d)^{\rho_1-\rho_2} (ad-bc)^{\rho_2} \nu^{(\omega-\rho_1-\rho_2)/2}. \]

\subsection{Branching laws in families}

Recall that for a Tate algebra $A$ endowed with an $n_0$-analytic character $\kappa_A \, : \, T(\ZZ_p) \rightarrow A^{\times}$ as above, and additionally with a character $\lambda \, : \, (1+\mathcal B_n)^{\times} \rightarrow A^{\times}$, we may define a special vector in $V_{G, \kappa_A}^{n-\an}$ (referred to as the ``krakenfish'' in \cite{LZ21-erl}) by the formula $\mathfrak K^{\lambda}(z) = \lambda(1+z)$.

The following lemma is analogous to \cite[Lemma 8.3.2]{LZ21-erl}, but recall that now the objects involved in the definition of $\mathcal T_n^{\diamond}$ are different.

\begin{lemma}
The function $\mathfrak K^{\lambda}$ is an eigenvector for $(\tau^{\sharp})^{-1} \mathcal T_n^{\diamond} \tau^{\sharp}$, with eigencharacter $w_{0,M}\kappa_A + (\lambda,-\lambda;0)$.
\end{lemma}

\begin{proof}
This follows from the same argument that has been done in \cite[\S8.3]{LZ21-erl} once we note that the element $\tau$ lies in the Siegel parabolic subgroup, and that only the projection to the Levi subgroup matters for the purpose of this computation.
\end{proof}

The following result is a straightforward consequence of the previous lemma.

\begin{proposition}
Pairing with $\mathfrak K^{\lambda}$ defines a homomorphism of $\mathcal T_n^{\diamond}$-representations \[ (\iota^{\sharp})^* (D_{G,\kappa_A}^{n-\an}) \longrightarrow D_{H,w_{01,M} \kappa_A+(\lambda,-\lambda;0)}^{n-\an}. \]
\end{proposition}

\subsection{Labelling of weights}

As above, let $(A,A^+)$ be a Tate algebra over $(\QQ_p,\ZZ_p)$. Given a weight $\nu_A \, : \, T(\ZZ_p) \rightarrow A^{\times}$ for some coefficient ring $A$, we may define $\kappa_A \, : \, T(\ZZ_p) \rightarrow A^{\times}$ by \[ \kappa_A = -w_{0,M} w_2(\nu+\rho) - \rho. \] If $\nu_A$ is $(\nu_1,\nu_2;\omega)$ for some $\nu_i,\omega \, : \, \ZZ_p^{\times} \rightarrow A^{\times}$, then $\kappa_A = (\nu_1,-2-\nu_2; \omega)$. Its Serre dual is $\kappa_A' = (\kappa_A + 2\rho_{\nc})^{\vee}$. This can be written as $(\nu_2-1,-3-\nu_1;c) = w_2(\nu_A+\rho)-\rho$.

\subsection{Sheaves on $G$}\label{sec:sheavesG}

Let $1 \leq n < t$ be integers. The following definition is just \cite[Def. 9.2.1]{LZ21-erl}.

\begin{definition}
The sheaf $\mathcal V_{G,\nu_A}^{n-\an}$ over $\mathcal U_n^G$ is given by the product \[ \mathcal V_{G,\nu_A}^{n-\an} = {}^{\mu} \mathcal M_{\HT,n,\Iw}^G \times^{M_{G,n}^{\square}} V_{G,\kappa_A}^{n-\an}. \] We define similarly another sheaf $\mathcal D_{G,\nu_A}^{n-\an}$ as \[ \mathcal D_{G,\nu_A}^{n-\an} = {}^{\mu} \mathcal M_{\HT,n,\Iw}^G \times^{M_{G,n}^{\square}} D_{G,(\kappa_A+2\rho_{\nc})}^{n-\an}. \]
\end{definition}

As discussed in loc.\,cit., the sheaves $\mathcal V_{G,\nu_A}^{n-\an}$ and $\mathcal D_{G,\nu_A}^{n-\an}$ are sheaves of $A$-modules compatible with base-change in $A$. If $A=\QQ_p$ and $\nu_A=(r_1,r_2;c)$ for integers $r_1 \geq r_2 \geq -1$, we have classical comparison maps \[ \mathcal V_{G,\kappa_1} \hookrightarrow \mathcal V_{G,\nu_A}^{n-\an}, \quad \mathcal D_{G,\nu_A}^{n-\an} \twoheadrightarrow \mathcal V_{G,(\kappa_A+2\rho_{\nc})^{\vee}} = \mathcal V_{G,\kappa_2}. \]
%where as usual we put $L_2$ for the irreducible $M_S$-representation with highest weight $L_2 \, : \, \lambda(r_2+2,-r_1)$.

\subsection{Sheaves on $H$}

We mimic the same definitions for $H$, using now the element $w_{01} \in W_H$ in place of $w_2$. Given an $n$-analytic character $\tau_A$, we define $\kappa_A^H = -\tau_A - 2\rho_H$, and we set \[ \mathcal V_{H,\diamond,\nu_A}^{n-\an} = {}^{\mu}\mathcal M_{\HT,n,\diamond}^H \times^{\mathcal T_n^{\diamond}} V_{H,\kappa_A^H}^{n-\an} \] and \[ \mathcal D_{H,\diamond,\tau_A}^{n-\an} = {}^{\mu}\mathcal M_{\HT,n,\Iw}^H \times^{\mathcal T_n^{\diamond}} D_{H,(\kappa_A^H +2\rho_H)}^{n-\an}. \]

\subsection{Branching for sheaves}

\begin{definition}
We say that $A$-valued, $n$-analytic characters $\nu_A$ and $\tau_A$ of $T(\ZZ_p)$ are compatible if $\nu_A = (\nu_1,\nu_2; \nu_1+\nu_2)$, $\tau_A = (\tau_1, \tau_2; \nu_1+\nu_2)$, for some characters $\nu_i,\tau_i$ of $\ZZ_p^{\times}$, and we have the relation \[ \tau_1-\tau_2 = \nu_1-\nu_2-2. \]
\end{definition}

If $\nu_A, \tau_A$ are compatible, then taking $\lambda = \nu_1 - \tau_1 = \nu_2 - \tau_2 + 2$, we obtain a homomorphism of $\mathcal T_n^{\diamond}$-representations \[ D_{G,(\kappa_A+2\rho_{\nc})}^{n-\an} \longrightarrow D_{H,-\tau_A}^{n-\an}. \]

\begin{proposition}
Pairing with $\mathfrak K^{\lambda}$ induces a morphism of sheaves over $\mathcal U_n^H$: \[ (\iota^{\sharp})^*(\mathcal D_{G,\nu_A}^{n-\an}) \longrightarrow \mathcal D_{H,\diamond,\tau_A}^{\an}. \] This morphism is compatible with specialisation in $A$, and if $A=\QQ_p$ and $\nu=(r_1,r_2;r_1+r_2)$, $\tau = (t_1,t_2; r_1+r_2)$ are algebraic weights with $r_1-r_2 \geq 0$ and $r_i,t_i \geq -1$, then this morphism is compatible with the map of finite dimensional sheaves $(\iota^{\sharp})^*(\mathcal V_{\kappa_2}) \rightarrow \mathcal V_{\tau}^H$, where $\mathcal V_{\kappa_2}$ is as in \cite[\S3]{LR2}.
\end{proposition}

\begin{proof}
This follows immediately from the results of Section \ref{sec:sheavesG}.
\end{proof}

\subsection{Locally analytic overconvengent cohomology}

We adopt the same definitions regarding cuspidal, locally analytic, overconvergent cohomology of \cite[\S9.5]{LZ21-erl}. In particular, \[ R\Gamma_{w,\an}^G(\nu_A, \cusp)^{-,\fs} = R\Gamma_{\mathcal I_{mn}^G} \Big( \mathcal U_n^G, \mathcal D_{G,\nu_A}^{n,-\an}(-D_G) \Big)^{-,\fs}, \] and similarly for the non-cuspidal version. Here, ``$-,\fs$'' is the finite-slope part for the dual Hecke operators $\mathcal{U}_{\Si}'$ and $\mathcal{U}'_{\Kl}$. This complex is independent of $m$, $n$ and $t$, and is concentrated in degrees $[0,1,2]$.

The previous discussion means that, given $\nu_A$ and $\tau_A$ satisfying $\tau_1-\tau_2 = \nu_1-\nu_2-2$, we have a morphism of complexes of $A$-modules
\begin{equation}\label{pull-back}
(\iota^{\sharp})^* : R\Gamma_{w,\an}^G(\nu_A,\cusp)^{-,\fs} \longrightarrow R \Gamma_{\mathcal Z_m^H}(\mathcal U_n^H, \mathcal D_{H,\diamond,\tau_A}^{n-\an}(-D_H)).
\end{equation}
The map $\iota_p$ induces in the same way a morphism of sheaves over $\mathcal U_n^H$ and an analogous morphism at the level of complexes of $A$-modules.

\subsection{Pairings and duality}

We may define \[ R\Gamma_{w_{01},\an}(\mathcal S_{H,\Iw}(p^t),\tau_A)^{+,\dag} = \lim_{\rightarrow} R \Gamma(\mathcal Z_{m,\Iw}^H(p^t), \mathcal V_{H,\Iw,\tau_A}^{\an}). \]

The following theorem will be crucially used in the definition of the $p$-adic $L$-function. It can be understood as a statement about cup products of overconvergent cohomology on $\mathcal S_H$.

\begin{theorem}\label{thm:pairing}
The cup product induces a pairing \[ H_{w_{01},\an}^1(\mathcal S_{H,\Iw}(p^t), \tau_A, \cusp)^{-,\dag} \times H_{w_{01},\an}^1(\mathcal S_{H,\Iw}(p^t), \tau_A)^{+,\dag} \longrightarrow A, \] whose formation is compatible with base-change in $A$, and it is also compatible with the Serre duality pairing on classical cohomology when $A=\QQ_p$ and $\nu,\tau$ are classical weights.
\end{theorem}

\begin{proof}
The map is defined using the pairing between the cohomology groups $H_{w_{01},\an}^1(\mathcal S_{H,\Iw}(p^t), \tau_A, \cusp)^{-,\dag}$ and $H_{w_{01},\an}^1(\mathcal S_{H,\Iw}(p^t), \tau_A)^{+,\dag}$. The result in the current form follows from \cite[Thm. 6.7.1]{boxerpilloni20}, from where it is clear that the pairing is compatible with Serre duality for each classical weight.
\end{proof}

\subsection{A K\"unneth formula for cohomology with support}

In order to define the $p$-adic $L$-function, we need to $p$-adically interpolate the cohomological pairing between $H^0$ and $H^1$. This may be regarded as a K\"unneth formula for cohomology with support.

\begin{proposition}\label{propo:pairing}
The cup product induces a pairing \[ H_{w_0,\an}^0(\mathcal S_{\GL_2,\Iw}(p^t), \tau_1)^{\dag} \times H_{w_1,\an}^1(\mathcal S_{\GL_2,\Iw}(p^t), \tau_2, \cusp)^{\dag} \longrightarrow H_{w_{01},\an}^1(\mathcal S_{H,\Iw}(p^t), \tau_A)^{-,\dag}, \] where $\tau_A = (\tau_1,\tau_2)$ is a weight for $H$.
\end{proposition}

\begin{proof}
This follows as in \cite[Thm. 9.6.2]{LZ21-erl} by the general theory of Boxer--Pilloni \cite[Thm. 6.7.1]{boxerpilloni20}.
\end{proof}

\section{The $p$-adic $L$-function}\label{sec:function}

In this section we discuss how to use higher Coleman theory to reinterpret the Harris--Su pairing, as discussed in \cite[\S3]{LR2}, in coherent cohomology over certain strata in suitable adic Shimura varieties. In particular, this analysis allows us to perform $p$-adic interpolation provided that there exist families of cohomology classes interpolating the different elements involved there. We implicitly use some of the results discussed along \cite[\S9,10]{LPSZ1}, as well as Novodvosky’s formula and its interpretation in coherent cohomology discussed in \cite{LR2}.

If not specified otherwise, $\pi$ and $\sigma$ are cohomological cuspidal automorphic representations of $\GSp_4$ and of $\GL_2$, defined over some number field $E$, both globally generic and unramified outside a certain finite set. Let $L$ be some $p$-adic field with an embedding from $E$.

\subsection{Tame test data}\label{sec:tame}

As in \cite[\S10.2]{LZ21-erl}, we fix the following data:
\begin{itemize}
\item $M_0,N_0$ are positive integers coprime to $p$ with $M_0^2 \mid N_0$, and $\chi_0$ is a Dirichlet character of conductor $M_0$ (valued in $L$).
\item $M_2, N_2$ are positive integers coprime to $p$ with $M_2 \mid N_2$, and $\chi_2$ is a Dirichlet character of conductor $M_2$ (valued in $L$).
\end{itemize}
As usual, we use the hat to denote the adelic counterpart of the characters. We shall consider automorphic representations $\pi$ of $G$ with conductor $N_0$ and central character $\hat \chi_0$ (up to twists by norm); here ``conductor'' is the analytic conductor of the associated degree 4 $L$-function, which always satisfies the divisibility $M_0^2 \mid N_0$. We assume similarly that the representation $\sigma$ of $\GL_2$ has conductor $N_2$ and character $\hat \chi_2$ (up to twists by norm).

Let $S$ denote the set of primes dividing $N_0N_2$. By tame test data we mean a pair $\gamma_S = (\gamma_{0,S},\Phi_S)$, where:
\begin{itemize}
\item $\gamma_{0,S} \in G(\QQ_S)$, where $\QQ_S = \prod_{\ell \in S} \QQ_{\ell}$;
\item $\Phi_S \in C_c^{\infty}(\QQ_S^2,L)$, lying in the $(\hat \chi_0 \hat \chi_2)^{-1}$-eigenspace for $\ZZ_S^{\times}$, where the action is as described in \cite[\S3]{LSZ17}.
\end{itemize}

We let $K_S$ be the quasi-paramodular subgroup (in the sense of \cite{okazaki}) of $G(\QQ_S)$ of level $(N_0,M_0)$, so that $\pi$ has one-dimensional invariants under $K_S$; and we let $\hat K_S$ be the open compact subgroup of $G(\QQ_S)$ defined in \cite[\S10.2]{LZ21-erl}. We also use analogous notations for $K^p$ and $\hat K^p$, the prime-to-$p$ part of the level and its adelic counterpart, respectively.

%As in loc.\,cit. we must consider a correction term $Z_S$. We let $\Pi$ and $\Sigma$ be the unitary twists of $\pi$ and $\sigma$ respectively, so that \[ L(\Pi \times \Sigma, s) = L(\pi \times \sigma, s + \frac{r_1+r_2+t_2}{2}). \] Then, we may define the zeta integral $Z(W_0,\Phi,W_2;s)$ and the zeta factors $Z_S(\pi \times \sigma, \gamma_S;s)$ and $Z_S(\pi \times \sigma, \sigma_S)$.

\subsection{$p$-adic families}

%As a general piece of notation, we use the conventions regarding $p$-adic families of \cite[\S10.4]{LZ21-erl}. In particular, for the group $G$, we consider $U \subset \mathcal W^2$ an open affinoid disc, and let $\hr_1, \hr_2 \, : \, \ZZ_p^{\times} \rightarrow \mathcal O(U)^{\times}$ be the universal characters associated to the two factors of $\mathcal W^2$. Let $\nu_U$ be the character $(\hr_1,\hr_2; \hr_1+\hr_2)$ of $T(\ZZ_p)$.
%Similar notations apply to the $\GL_2$ factors, where we use $\htt$ and $\hl$ for the corresponding characters.

%As a general piece of notation, we use the conventions regarding $p$-adic families of \cite[\S10.4]{LZ21-erl}. In particular, we consider $U \subset \mathcal W^2$ an open affinoid disc, and let $\hr_1, \hr_2 \, : \, \ZZ_p^{\times} \rightarrow \mathcal O(U)^{\times}$ be the universal characters associated to the two factors of $\mathcal W^2$. Let $\nu_U$ be the character $(\hr_1,\hr_2; \hr_1+\hr_2)$ of $T(\ZZ_p)$.
%Similar notations apply to the $\GL_2$ factors, where we use $\hc_1$ and $\hc_2$ for the corresponding characters.

As a general piece of notation, we use the conventions regarding $p$-adic families of \cite[\S10.4]{LZ21-erl}. In particular, we consider $U \subset \mathcal W^2$ an open affinoid disc, and let $\hr_1, \hr_2 \, : \, \ZZ_p^{\times} \rightarrow \mathcal O(U)^{\times}$ be the universal characters associated to the two factors of $\mathcal W^2$. Let $\nu_U$ be the character $(\hr_1,\hr_2; \hr_1+\hr_2)$ of $T(\ZZ_p)$.

\begin{definition}
A {\it family of automorphic representations} $\underline{\pi}$ of tame level $N_0$ and character $\chi_0$ over an open affinoid disc $U$ is the data of a finite flat covering $\tilde U \rightarrow U$ and a homomorphism $\tilde U \rightarrow \mathcal E$ lifting the inclusion $U \hookrightarrow \mathcal W$, such that the following conditions hold:
\begin{enumerate}
\item[(a)] $\tilde U$ is 2-dimensional and smooth;
\item[(b)] the restriction of $H^k(\mathcal M_{\text{cusp},w_j}^{\bullet,-,fs})$ to $\tilde U$ is zero if $j+k \neq 3$, and the sheaves $S^k(\underline{\pi}) = H^k(\mathcal M_{\text{cusp},w_{3-k}}^{\bullet,-,fs})$ are either free over $\mathcal O(\tilde U)$ of rank 1 for all $k$ (general-type), or free of rank 1 for $k=1,2$ and zero for $k=0,3$ (Yoshida-type);
\item[(c)] the centre of $G(\AA_{\fin}^p)$ acts on $S^k(\underline{\pi})$ by $|\cdot|^{-(\hr_1+\hr_2)} \hat \chi_0$.
\end{enumerate}
\end{definition}

Assume that the representation $\pi$ can be interpolated along a finite-slope overconvergent $p$-adic family of automorphic representations $\underline{\pi}$ over the open affinoid disc $U$ introduced at the beginning of the section. (Given any cohomological $\pi$ with sufficiently small slope at $p$, the theory of eigenvarieties guarantees that we can always find a sufficiently small open disc around the weight of $\pi$ such that this holds.) Let $S^1(\underline{\pi}) = H^1(\mathcal M_{\text{cusp},w_2}^{\bullet,-,fs})$ be the sheaf introduced in \cite[Def. 10.4.1]{LZ21-erl}, which is free of rank 1 according to the definition we have made. We shall then choose a basis $\underline{\xi}$ of that space. Since the spaces of higher Coleman theory have an action of $G(\AA_{\fin}^p)$, we can make sense of $\gamma_{0,S} \cdot \underline{\xi}$ as a family of classes at tame level $\hat K^p$, which is still an eigenfamily for the Hecke operators away from $S$.
%This includes a finite flat covering $\tilde U \rightarrow U$, which allows us to consider the notion of {\it nice point} for ${\underline \pi}$.

\begin{definition}
We say a point $P \in U(L)$ is nice for ${\underline \pi}$ if the weight of $P$ is $(r_1,r_2) \in U \cap \ZZ^2$ with $r_1 \geq r_2 \geq 0$ and the specialisation at $P$ of the system of eigenvalues $\lambda_{\underline{\pi}}^-$ attached to the family $\underline{\pi}$ is the character of a $p$-stabilised automorphic representation $\pi_P$, which is cuspidal, globally generic, and has conductor $N_0$ and character $\chi_0$.
\end{definition}

This implies that the fibre of $S^1(\underline{\pi})$ at $P$ maps isomorphically to the $\pi_P$-eigenspace in the classical $H^1(K^p,\kappa_1(\nu),\cusp)$; in particular, this eigenspace is 1-dimensional. By the classicality theorems for higher Coleman theory, given a family ${\underline{\pi}}$, all specialisations of integer weight $(r_1,r_2)$ with $r_1-r_2$ and $r_2$ sufficiently large relative to the slope of ${\underline{\pi}}$ will be nice; and if ${\underline{\pi}}$ is ordinary, it suffices to assume that $r_1-r_2 \geq 3$ and $r_2 \geq 0$.

We can consider analogous objects for $\GL_2$. In particular, we may choose a disc $U' \subset \mathcal W$ and a finite-slope overconvergent $p$-adic family of modular eigenforms $\mathcal G$ over $U'$ (of weight $\hl+2$ where $\hl$ is the universal character associated to $U'$). In particular, we also impose that the corresponding spaces $S^0(\mathcal G)$ and $S^1(\mathcal G)$ are free of rank one. Then, we say a point $Q \in U'$ is nice for $\mathcal G$ if it lies above an integer $\ell \in U' \cap \ZZ_{\geq 0}$, and the specialisation of $\mathcal G$ at $Q$ is a classical form. We further require that the fibre of $S^1(\underline{\sigma})$ at $Q$ maps isomorphically to the $\sigma_Q$-eigenspace in the classical $H^1$ (and in particular, this eigenspace is 1-dimensional).
We write $\sigma_{\ell}$ for the corresponding automorphic representation, with the normalisations of loc.\,cit. As before, we shall take a basis $\underline{\eta}$ of $S^1(\underline{\sigma})$.

\begin{remark}
The inequalities defining region $(D)$ automatically imply that we are not dealing with non-cohomological weights, and hence we do not need to consider an \'etale covering, as it was the case for region $(F)$.
\end{remark}

\subsection{Construction of the imprimitive $p$-adic $L$-function}

%Recall from \cite[\S3]{LR2} that given a weight $c_2$ cuspidal automorphic form $\sigma$ in $\GL_2$, one can attach to it an element $\eta \in H^1(X_{\GL_2}, \mathcal V_{2-c_2})$. Further, to define the $p$-adic $L$-function we need to consider an auxiliary weight $c_1$ modular form, which can be naturally interpreted in $H^0(X_{\GL_2},\mathcal V_{c_1})$.

%In the case where the element in $H^0(X_{\GL_2},\mathcal V_{c_1})$ is an Eisenstein series we recover the $p$-adic $L$-function for $\GSp_4 \times \GL_2$, while in the cuspidal case we have the $p$-adic $L$-function for $\GSp_4 \times \GL_2 \times \GL_2$. In this article we focus just on the former situation.

We refer to \cite[\S7.4]{LPSZ1} for the construction of the $p$-adic family of Eisenstein series $\mathcal E^{\Phi^{(p)}}(0,\htt-1)$, which depends on a prime-to-$p$ Schwartz function $\Phi^{(p)}$. According to \cite[Prop. 10.1.2]{LZ21-erl}, it is an overconvergent cusp form of weight $\htt$ which may be understood as an element in the overconvergent $H^0$ of $\mathcal S_{\GL_2,\Iw}$. As discussed at the end of \S10.2 of loc.\,cit., $\Phi^{(p)}$ and $\Phi_S$ agree up to multiplication by the characteristic function of $\hat{\ZZ}^{S \cup \{p\}}$.

For our construction, we need to recall the pairing \[ H_{w_0,\an}^0(\mathcal S_{\GL_2,\Iw}(p^2), \tau_1)^{\dag} \times H_{w_1,\an}^1(\mathcal S_{\GL_2,\Iw}(p^2), \tau_2, \cusp)^{\dag} \longrightarrow H_{w_{01},\an}^1(\mathcal S_{H,\Iw}(p^2), \tau_A)^{-,\dag} \] introduced in Proposition \ref{propo:pairing}. From now on, let $A = \mathcal O(U \times U')$. Next, we can consider \[ \mathcal E^{\Phi^{(p)}}(0,\htt-1) \boxtimes G(\chi_2^{-1}) \underline{\eta} \in H_{w_{01},\an}^1(\mathcal S_{H,\Iw}(p^2), \tau_A)^{+,\dag}, \] where $\htt = \hr_2-\hr_1+\hl-2$ and the tame level is taken to be $H \cap \hat K^p$.

\begin{remark}
 The Gauss sum can be normalised away by re-scaling $\underline{\eta}$ if the coefficient field $L$ contains a root of unity of order $M_2$, but we do not assume this here.
\end{remark}

\begin{definition}\label{def:p-adic}
We let $\mathcal L_{p,\gamma_S}^{\imp}(\pi \times \sigma; \underline{\xi}; \underline{\eta})$ denote the element of $A$ defined by \[ \langle (\iota^{\sharp})^* ( \gamma_{0,S} \cdot \underline{\xi} ), \mathcal E^{\Phi^{(p)}}(0,\htt-1) \boxtimes G(\chi_2^{-1}) \underline{\eta} \rangle, \] where we are using the pairing of Theorem \ref{thm:pairing}.
\end{definition}

This is a three-variable $p$-adic $L$-function, where we may vary the weights $(r_1,r_2)$ and we keep the linear condition in terms of $(r_1,r_2,\ell,t)$, namely $\htt = \hr_2-\hr_1+\hl-2$ (alternatively, $s=\frac{\ell-r_1+r_2-2}{2}$).

%\ORnote{Clarify yet how many variables we have: we will keep $c_1-c_2$ fixed, but I don't know if we can vary $r_1$ and/or $r_2$. Further, we need to modify the first factor accordingly to later recover multiplication by $t$ in the $\GL_2$-factors.}

\begin{definition}
\begin{itemize}
\item We say a point $(P,Q)$ of $U \times U'$ is nice if $P=(r_1,r_2)$ and $Q=(\ell)$ are integer points, with $P$ nice for $\underline{\pi}$ and $Q$ nice for $\underline{\sigma}$.
\item We say that $(P,Q)$ is nice critical if we also have $\ell \leq  r_1-r_2+1$ (the specialisation $t$ of $\htt$ at $(P,Q)$ is $\geq -1$).
\item If instead we have $r_1-r_2 \leq \ell - 2 \leq r_1$, we say that $P$ is nice geometric.
\end{itemize}
\end{definition}

\subsection{The correction term $Z_S$}\label{sec:ZS}

This section introduces a correction term $Z_S$ which depends on the choice of local data, and which will arise in the interpolation property of the $p$-adic $L$-function. Its definition depends on certain Whittaker models properly introduced in \cite[\S6]{LR2}; since this will have a minor relevance in this work, we just refer the interested reader to our previous paper. In particular, with the notations introduced in loc.\,cit., we may consider the integral $Z(W,\Phi_1,W^{(\ell)};s)$.

We shall set \[ Z_S(\pi \times \sigma, \gamma_S; s) = \frac{Z(\gamma_{0,S} \cdot W_0^{\text{new}}, \Phi_S, W_2^{\text{new}}; s)}{G(\chi_2^{-1}) \prod_{\ell \in S} L(\pi_{\ell} \times \sigma_{\ell}, s)}, \] and \[ Z_S(\pi \times \sigma, \gamma_S) = Z_S(\pi \times \sigma, \gamma_S; 1 + \frac{t}{2}), \] where $t = r_2-r_1-2+\ell$, as usual, and $G(\chi_2^{-1})$ is the Gauss sum of the character $\chi_2$. Note that for any given $\pi$ and $\sigma$, one can choose $\gamma_S$ such that $Z_S(\pi \times \sigma, \gamma_S;s) \neq 0$ (this follows from the definition of the $L$-factor as a GCD of local zeta-integrals).

%\subsection{Choosing the global data}

%The global Whittaker transform, given by integrating automorphic forms over the compact quotient $N(\QQ) \backslash N(\AA)$, where $N$ is the upper-triangular unipotent subgroup of $\GSp_4$, gives a canonical isomorphism \[ \pi \cong \cW(\pi) \cong \otimes_v' \cW(\pi_v). \]

%For all finite places $v$, the space $\cW(\pi_v)$ has a normalised new-vector $w_v^{\new}$. Hence, given $w_p \in \cW(\pi_p)$, we can consider the global Whittaker function \[ w_{\infty} \cdot w_p \cdot \prod_{v \notin \{p,\infty\}} \gamma_v w_v^{\new}, \] where $\gamma_v$ is an arbitrary element of $\GSp_4(\QQ_v)$ which is the identity if $v$ is unramified, and $w_{\infty}$ is the standard Whittaker function at $\infty$. The theory for the $\sigma_i$ is analogous, with the standard Whittaker function being the complex exponential.

%Now, we shall consider the periods \[ \mathcal P \Big( \varphi_{\infty}^{\Iw} \times \sigma_1^{[p]} \times \delta^t(\sigma_2) \Big), \] where $t = \frac{c_1-c_2-r_1+r_2-2}{2}$. (Define suitably the periods appearing here!!)

\subsection{Interpolation property}

We choose a $\bar{\QQ}$-basis $\xi$ of the new subspace of $H^1(\pi_{\fin})$, where $H^1(\pi_{\fin})$ is the copy of $\pi_{\fin}$ appearing in the degree 1 coherent cohomology of the Siegel Shimura variety. Analogously, we also choose a $\bar{\QQ}$-basis $\eta$ of the new subspace of $H^1(\sigma_{\fin})$. We write $S^1(\pi,L)$ for the cohomology with $L$-coefficients, which is an $L$-vector space.

The element $\xi \otimes \eta$ is an explicit multiple of the standard Whittaker function, and the corresponding multiple defines a complex period $\Omega_{\infty}(\pi,\sigma) \in \CC^{\times}$.

\begin{definition}\label{def:periods}
Given non-zero $\xi \in S^1(\underline{\pi},L)$ and $\eta \in S^1(\underline{\sigma},L)$, we define periods $\Omega_p(\underline{\pi},{\underline \sigma}) \in L^{\times}$ and $\Omega_{\infty}(\underline{\pi},{\underline \sigma}) \in \CC^{\times}$ as in \cite[\S10.2]{LPSZ1}. (These periods do depend on the choices of $\xi$ and $\eta$ up to multiplication by $L^{\times}$, but we drop that dependence from the notation). We write $\Omega_p(\pi_P,\sigma_Q)$ and $\Omega_{\infty}(\pi_P,\sigma_Q)$ for the specialization of the periods at $(P,Q)$.
\end{definition}

More precisely, the space of Whittaker-$E$-rational classes is exactly $\Omega_{\infty}(\pi,\sigma) \cdot H^1(\pi_{\fin}) \otimes H^1(\sigma_{\fin})$, for a nonzero constant $\Omega_{\infty}(\pi,\sigma) \in \CC^{\times}$, that we call the Whittaker period.

In the following result we establish the interpolation property for the $p$-adic $L$-function; observe that the algebraicity of the right-hand side was the main result of \cite{LR2}. Recall the degree-8 Euler factor $\mathcal{E}^{(d)}$ of \cite{LZvista}, that we call $\mathcal{E}^{(D)}$ in the $\GSp_4 \times \GL_2$ setting of \cite{LR2}.

\begin{theorem}\label{thm:d}
The $p$-adic $L$-function $\mathcal L_{p,\gamma_S}^{\imp}(\underline{\pi} \times \underline{\sigma})$ has the following interpolation property: if $(P,Q)$ is nice critical, with $P$ of weight $(r_1,r_2)$ and $Q$ of weight $\ell$, then \[ \frac{\mathcal L_{p,\gamma_S}^{\imp}(\underline{\pi} \times \underline{\sigma})(P,Q)}{\Omega_p(\pi_P,\sigma_Q)} = Z_S(\pi_P \times \sigma_Q, \gamma_S) \cdot \mathcal E^{(D)}(\pi_P \times \sigma_Q) \cdot \frac{G(\chi_2^{-1}) \Lambda(\pi_P \times \sigma_Q, \frac{\ell-r_1+r_2-2}{2})}{\Omega_{\infty}(\pi_P, \sigma_Q)}, \] where $\Lambda(\pi_P \times \sigma_Q, s)$ is the completed (complex) $L$-function and $G(\chi_2^{-1})$ is the Gauss sum of $\chi_2^{-1}$.
\end{theorem}

\begin{proof}
By construction, we have
\[ \mathcal L_{p,\gamma_S}^{\imp}(\underline{\pi} \times \underline{\sigma})(P,Q) = G(\chi_2^{-1}) \langle (\iota^\sharp)^* ( \gamma_{0,S} \cdot \xi_P ), \mathcal E^{\Phi^{(p)}}(0,t-1) \boxtimes \eta_Q \rangle. \]

Along the region given by $\ell-t = r_1-r_2+2$, this expands as the product of $G(\chi_2^{-1}) \Lambda(\pi_P \times \sigma_Q, \frac{t}{2})$ and a product of normalised local zeta-integrals. The local zeta-integral at $p$ has been evaluated in \cite[\S7]{LR2} and gives the desired Euler factor. The product of zeta-integrals at the bad primes is by definition $G(\chi_2^{-1}) Z_S(...)$.

%Hence, and following Theorem \ref{pair}, this recovers \[ G(\chi_2^{-1}) \langle \xi_P, \hat \iota_*(\mathcal E^{\Phi^{(p)}}(0,c_1-1) \boxtimes \eta_Q) \rangle = \frac{G(\chi_2^{-1})}{(2\pi i)^3} \int_{\RR^{\times} H(\QQ) \backslash H(\AA)} F_{\xi}(v_{c_1,c_2}) f(h_1)g(h_2) \, dh, \] where $v_{c_1,c_2}$ is the standard basis vector of $\tau_2$ of weight $(-c_1,-c_2)$. \ORnote{Pending to write this in a suitable way to relate it with $L$-values.}
\end{proof}

\begin{remark}
Taking into account the discussions of \cite[Rmk. 7.14]{LR2}, it is possible to use this same method to get an improved $p$-adic $L$-function where the interpolation property involves a degree seven Euler factor. Further, following the recent work of Graham, Pilloni and Rodrigues Jacinto \cite{GPR}, it should be possible to extend the previous construction to a $p$-adic $L$-function in all four variables.
\end{remark}

\section{A conjectural reciprocity law}

\subsection{Slope conditions}

 We now recall various notions of \emph{slope} associated to automorphic representations of $G$ and $H$. Given a cohomological automorphic representation $\pi$ of $G$ such that $\pi_p$ has non-zero invariants under $\Iw_G$ (with a chosen embedding of its coefficient field into $\overline{\QQ}_p$), and a simultaneous eigenspace in $(\pi_p)^{\Iw_G}$ for the operators $\mathcal{U}_{\Si}'$ and $\mathcal{U}_{\Kl}'$, we define the \emph{slope} of this eigenspace to be the pair of rational numbers $\lambda(\mathcal{U}_{\Si}')$, $\lambda(\mathcal{U}_{\Kl}')$ which are the valuations of the eigenvalues for these operators. These slopes play a central role in the classicity criteria of \cite{boxerpilloni20}. (One can use either the usual Hecke operators $\mathcal{U}_{\Si}$ and $\mathcal{U}_{\Kl}$, or the dual operators $\mathcal{U}_{\Si}'$ and $\mathcal{U}_{\Kl}'$, since the same eigenvalues appear for both choices.)

 If $\lambda(\mathcal{U}_{\Si}') = 0$ we say the eigenspace is \emph{Siegel-ordinary}, and similarly \emph{Klingen-ordinary}; and we say that $\pi$ is Borel-ordinary at $p$ if it is both Siegel and Klingen ordinary. The condition of being Siegel ordinary at $p$ may be rephrased by requiring that $v_p(\alpha)=0$, and being Klingen ordinary is equivalent to $v_p(\alpha \beta ) = r_2+1$.

For a cuspidal automorphic representation $\sigma$ of $\GL_2$, write $\mathfrak a, \, \mathfrak b$ for the Hecke parameters of $\sigma_p$ (that is, the parameters corresponding to the action of the dual Hecke operators). We adopt the convention that $v_p(\mathfrak a) \leq v_p(\mathfrak b)$ and say that $\sigma$ is Borel-ordinary at $p$ (with respect to $v$) if $v_p(\mathfrak a) = 0$.

The hypotheses of the classicity theorems in \cite{boxerpilloni20} require two (slightly different) notions of ``small slope'', which we shall make explicit here. We consider the Hecke operators with the previously discussed normalizations acting on the cohomology of the sheaves $\mathcal V_{\kappa}$. Thus each operator is ``minimally integrally normalised" acting on the classical cohomology (slopes are $\geq 0$). Write $K^p$ for some fixed choice of open compact away from $p$. Then \cite[Conj. 5.9.2]{boxerpilloni20} predicts lower bounds for the slopes of the Hecke operators acting on the overconvergent cohomology complexes $R\Gamma_w(K^p,\kappa)^{\pm}$ and $R\Gamma(K^p,\kappa,\cusp)^{\pm}$, whose precise definitions are given in loc.\,cit.; and there are similar conjectures for the locally-analytic cohomology complexes.

We compute for various elements $w \in W_G$ the character $w^{-1}w_G^{\max}(\kappa+\rho)-\rho$, and how it pairs with the anti-dominant cocharacters $\diag(1,1,x,x,)$ and $\diag(1,x,x,x^2)$ defining the operators $\mathcal U_{\Si}'$ and $\mathcal U_{\Kl}'$. We take $\kappa = \kappa_2 = (r_2-1,-r_1-3;r_1+r_2)$, and subtract $r_2$ from all entries in the bottom row since this is our normalising constant for $\mathcal U_{\Kl}'$. Following the approach of \cite[\S5.11]{boxerpilloni20}, we summarize the conjectural slope bounds in the following table.

\begin{center}
\begin{tabular}{p{2cm}|p{2cm}p{2cm}p{2cm}p{2cm}}
$w =$ & $\id$ & $w_1$ & $(w_2)$ & $w_3$ \\
\hline
$\mathcal U_{\Si}'$ & $r_1+2$ & $0$ & $(0)$ & $r_2+1$ \\
$\mathcal U_{\Kl}'$ & $r_1-r_2+1$ & $r_1-r_2+1$ & $(0)$ & $0$ \\
\end{tabular}
\end{center}

We do not know this conjecture in full, but from \cite[Thm. 5.9.6, Thm. 6.8.3]{boxerpilloni20}, we do know a weaker statement in which we replace $w^{-1}w_G^{\max}(\kappa_2+\rho)-\rho$ with $w^{-1}w_G^{\max}\kappa_2$. This gives the following bounds:

\begin{center}
\begin{tabular}{p{2cm}|p{2cm}p{2cm}p{2cm}p{2cm}}
$w =$ & $\id$ & $w_1$ & $(w_2)$ & $w_3$ \\
\hline
$\mathcal U_{\Si}'$ & $r_1+2$ & $-1$ & $(-1)$ & $r_2-2$ \\
$\mathcal U_{\Kl}'$ & $r_1-r_2+1$ & $r_1-r_2+1$ & $(-3)$ & $-3$ \\
\end{tabular}
\end{center}

The following proposition discusses the conditions of ``small slope" and ``strictly small slope". The reason for introducing different slope conditions is that the conditions needed to obtain a vanishing theorem are not the same as those needed to obtain classicality theorems; further, there are different kinds of control theorems requiring distinct sets of hypotheses.

\begin{proposition}
For the weight $\kappa_2 = (r_2-1,-r_1-3; r_1+r_2)$, with $r_1 \geq r_2 \geq 0$, we have the following.
\begin{itemize}
\item The ``small slope" condition $(-,\ssa^M(\kappa_2))$ is \[ \lambda(\mathcal U'_{\Si}) < r_1+2, \quad \lambda(\mathcal U'_{\Kl}) < r_1-r_2+1. \]
\item The ``strictly small slope" condition $(-,\sss^M(\kappa_2))$ is \[ \lambda(\mathcal U'_{\Si}) < r_1+2, \quad \lambda(\mathcal U'_{\Kl}) < r_1-r_2-2. \]
\end{itemize}
\end{proposition}

\begin{proof}
This follows from the previous tables.
\end{proof}

\subsection{Ordinary filtrations at $p$}

Along the rest of this section we assume that $\pi$ is both Klingen and Siegel ordinary, and that $\sigma$ is Borel ordinary. This is done just with the purpose of simplifying notations; similar conjectures can be formulated in the more general strictly-small-slope setting, but one needs to use the theory of $(\varphi,\Gamma)$-modules over the Robba ring (rather than actual subrepresentations of Galois representations). Further, one expects to be able to formulate integral refinements in the ordinary setting, using Coleman maps instead of the Perrin-Riou map, and making $\mathcal L_p$ a $p$-adic measure instead of just a distribution. We begin by discussing the slope conditions.

Associated with the family $\underline{\pi}$ we have a family of Galois representations $V(\underline{\pi})$, which is a rank 4 $\mathcal O(U)$-module with an action of $\Gal(\overline{\QQ}/\QQ)$, unramified outside $pN_0$ and with a prescribed trace for $\Frob_{\ell}^{-1}$, when $\ell \nmid pN_0$. The Galois representation $V(\underline{\pi})$ has a decreasing filtration by $\mathcal O(U)$-submodules stable under $\Gal(\overline{\QQ}_p/\QQ_p)$. Borrowing the notations from \cite[\S11]{LZ21-erl}, we write $\mathcal F^i V(\underline{\pi})$ for the codimension $i$ subspace, and similarly for its dual $V(\underline{\pi})^*$. Similarly, there is a 2-step filtration for $V(\underline{\sigma})$. See e.g. \cite[\S9.1]{LZvista} for a precise account of the feature of the different filtrations involved in the picture.

\begin{definition}
We set \[ \mathbb V^* = V(\underline{\pi})^* \otimes V(\underline{\sigma}) (-1-\hr_1); \] and we let \[ \mathcal F^{(D)} V(\underline{\pi} \times \underline{\sigma})^* = (\mathcal F^1 V(\underline{\pi})^* \otimes \mathcal F^1 V(\underline{\sigma})^*) + (\mathcal F^3 V(\underline{\pi})^* \otimes V(\underline{\sigma})^*) \] and \[ \mathcal F^{(E)} V(\underline{\pi} \times \underline{\sigma})^* = (\mathcal F^1 V(\underline{\pi})^* \otimes \mathcal F^1 V(\underline{\sigma})^*) + (\mathcal F^2 V(\underline{\pi})^* \otimes V(\underline{\sigma})^*). \] For a nice weight $(P,Q)$ we write $\mathbb V_{P,Q}^*$ for the specialisation of $\mathbb V^*$ at $(P,Q)$, so $\mathbb V_{P,Q}^* = V(\pi_P)^* \otimes V(\sigma_Q)^*(-1-r_1)$ if $P=(r_1,r_2)$.
\end{definition}

In particular, $\mathcal F^{(E)}$ has rank 5, $\mathcal F^{(D)}$ has rank 4, and the quotient $\Gr^{(e/d)}$ is isomorphic to \[ \Gr^{(E/D)} \cong (\Gr^2 V(\underline{\pi})^*) \otimes (\Gr^0 V(\underline{\sigma})^*) (-1-\hr_1). \] Observe that while in \cite{LZ21-erl} the quotient we were interested in was $(\Gr^1 V(\underline{\pi})^*) \otimes (\Gr^1 V(\underline{\sigma})^*)$, here we are using a different step of the filtration. This is because in loc.\,cit. we were comparing the regions $(e)$ and $(f)$, while here the contrast is between $(e)$ and $(d)$, so the filtrations involved in each factor are different.

\subsection{$p$-adic periods and $p$-adic Eichler--Shimura isomorphisms}

The representations $\Gr^2 V(\underline{\pi})(-2-\hr_1)$ and $\Gr^0 V(\underline{\sigma})$ are unramified, and hence crystalline as $\mathcal O(U)$ (resp. $\mathcal O(U')$)-linear representations. Since $\DD_{\cris}(\QQ_p(1))$ is canonically $\QQ_p$, we can therefore define $\DD_{\cris}(\Gr^{(e/d)} \mathbb V^*)$ to be an alias for the rank 1 $\mathcal O(U \times U')$-module \[ \DD_{\cris}(\Gr^2 V(\pi)^*(-2-\hr_1)) \hat \otimes \DD_{\cris}(\Gr^0 V(\underline{\sigma})^*). \] We can then define a Perrin-Riou big logarithm for $\Gr^{(e/d)} \mathbb V^*$, which is a morphism of $\mathcal O(U \times U')$-modules \[ \mathcal L^{\PR} : H^1(\QQ_p, \Gr^{(e/d)} \mathbb V^*) \longrightarrow \DD_{\cris}(\Gr^{(e/d)} \mathbb V^*). \] For nice geometric weights $P$, this specialises to the Bloch--Kato logarithm map, up to an Euler factor; and for nice critical weights is specialises to the Bloch--Kato dual exponential.

Let $P$ be a nice weight. There is an {\it Eichler--Shimura} isomorphism \[ \ES_{\pi_P}^1 : S^1(\pi_P, L) \cong \DD_{\cris}(\Gr^2(V(\pi_P))). \]

Similarly, for $\GL_2$ we have an isomorphism \[ \ES_{\sigma_Q}^1 : S^1(\sigma_Q, L) \cong \DD_{\cris}(\Gr^0 V(\sigma_Q)). \] In this case, the existence of a comparison in families is known after Kings--Loeffler--Zerbes \cite{KLZ17}, that is, there exists an isomorphism of $\mathcal O(U')$-modules \[ \ES_{\underline{\sigma}}^1 : S^1(\underline{\sigma}) \cong \DD_{\cris}(\Gr^0 V(\underline{\sigma})) \] interpolating the isomorphism $\ES_{\sigma_Q}^1$ for varying $Q$, where $\mathcal S^1(\underline{\sigma})$ is the $\mathcal O(U')$-module spanned by $\underline{\eta}$.

\subsection{Euler system classes}

Suppose that the character $\chi_0 \chi_2$ is non-trivial. Then, by the results of \cite{HJS20}, associated to the data $\gamma_S$, we have a family of cohomology classes \[ \hz_m(\underline{\pi} \times \underline{\sigma}, \gamma_S) \in H^1(\QQ(\mu_m), \mathbb V^*), \] for all square-free integers coprime to some finite set $T$ containing both $p$ and the ramified primes. The image of $\hz_m(\underline{\pi} \times \underline{\sigma},\gamma_S)$ under localisation at $p$ lands in the image of the injective map from the cohomology of $\mathcal F^{(E)} \mathbb V^*$ and we can therefore make sense of \[ \mathcal L^{\PR}(\hz_m(\underline{\pi} \times \underline{\sigma}),\gamma_S) \in \DD_{\cris}(\Gr^{(e/f)} \mathbb V^*). \]

In this setting, we expect the following result.

\begin{conjecture}\label{conj:rec-law}
Under the running assumptions, the equality \[ \big \langle \mathcal L^{\PR}(\hz_1(\underline{\pi} \times \underline{\sigma}, \gamma_S))(P,Q), \ES_{\pi_P}^1(\xi_P) \otimes \ES_{\sigma_Q}^1(\eta_Q) \big \rangle = \mathcal L_{p,\gamma_S}^{\imp}(\underline{\pi} \times \underline{\sigma})(P,Q) \] holds for all $(P,Q)$ in the geometric range.
\end{conjecture}

The main difficulty for proving the theorem following an analogous strategy to the case of region $(F)$ is the lack of semistable models for the different Shimura varieties involved in this picture (Siegel level). We hope that a better understanding of higher Coleman theory following the new results of Boxer and Pilloni could lead to a proof of the previous conjecture.

\let\MR\undefined
\newlength{\bibitemsep}
\setlength{\bibitemsep}{0.75ex plus 0.05ex minus 0.05ex}
\newlength{\bibparskip}
\setlength{\bibparskip}{0pt}
\let\oldthebibliography\thebibliography
\renewcommand\thebibliography[1]{%
 \oldthebibliography{#1}%
 \setlength{\parskip}{\bibparskip}%
 \setlength{\itemsep}{\bibitemsep}%
}
\providecommand{\noopsort}[1]{\relax} % dummy function for sorting

\providecommand{\bysame}{\leavevmode\hbox to3em{\hrulefill}\thinspace}
\providecommand{\MR}[1]{%
 MR \href{http://www.ams.org/mathscinet-getitem?mr=#1}{#1}.
}
\providecommand{\href}[2]{#2}
\newcommand{\articlehref}[2]{\href{#1}{#2}}


\begin{thebibliography}{GPRJ23}

\bibitem[Aga07]{agarwal}
\textsc{M.~K. Agarwal}, \emph{p-adic {L}-functions for {GS}p(4) x {GL}(2)},
  ProQuest LLC, Ann Arbor, MI, 2007, Thesis (Ph.D.)--University of Michigan.
  \MR{2711014}

\bibitem[BP20]{boxerpilloni20}
\textsc{G.~Boxer} and \textsc{V.~Pilloni},
  \articlehref{https://perso.ens-lyon.fr/vincent.pilloni/HigherColeman.pdf}{\emph{Higher
  {C}oleman theory}}, preprint, 2020.

\bibitem[Gra24]{graham21}
\textsc{A.~Graham},
  \articlehref{https://doi.org/10.2140/ant.2024.18.1117}{\emph{On the
  {$p$}-adic interpolation of unitary {F}riedberg-{J}acquet periods}}, Algebra
  Number Theory \textbf{18} (2024), no.~6, 1117--1188. \MR{4740094}

\bibitem[GPRJ23]{GPR}
\textsc{A.~Graham, V.~Pilloni,} and \textsc{J.~Rodrigues~Jacinto},
  \articlehref{http://arxiv.org/abs/2311.14438}{\emph{{$p$}-adic interpolation
  of gauss--manin connections on nearly overconvergent modular forms and
  {$p$}-adic {$L$}-functions}}, preprint, 2023, \path{arXiv:2311.14438}.

\bibitem[HJS20]{HJS20}
\textsc{C.-Y. Hsu, Z.~Jin,} and \textsc{R.~Sakamoto}, \emph{An {E}uler system
  for {$\operatorname{GSp}(4)\times\operatorname{GL}(2)$}}, preprint, 2020.

\bibitem[KLZ17]{KLZ17}
\textsc{G.~Kings, D.~Loeffler,} and \textsc{S.~L. Zerbes},
  \articlehref{http://doi.org/10.4310/CJM.2017.v5.n1.a1}{\emph{{R}ankin--{E}isenstein
  classes and explicit reciprocity laws}}, Cambridge J. Math. \textbf{5}
  (2017), no.~1, 1--122. \MR{3637653}

\bibitem[Liu23]{liu23}
\textsc{Z.~Liu}, \articlehref{http://arxiv.org/abs/2308.08533}{\emph{P-adic
  {L}-functions for {$\GSp(4) \times \GL(2)$}}}, preprint, 2023,
  \path{arXiv:2308.08533}.

\bibitem[LPSZ21]{LPSZ1}
\textsc{D.~Loeffler, V.~Pilloni, C.~Skinner,} and \textsc{S.~L. Zerbes},
  \articlehref{http://doi.org/10.1215/00127094-2021-0049}{\emph{Higher {H}ida
  theory and {$p$}-adic {$L$}-functions for {$\operatorname{GSp}(4)$}}}, Duke
  Math. J. \textbf{170} (2021), no.~18, 4033--4121.

\bibitem[LR24]{LR2}
\textsc{D.~Loeffler} and \textsc{O.~Rivero},
  \articlehref{https://doi.org/10.1093/qmath/haae016}{\emph{Algebraicity of
  {$L$}-values for {$\rm GSp_4 \times GL_2$} and {$\rm GSp_4 \times GL_2 \times
  GL_2$}}}, Q. J. Math. \textbf{75} (2024), no.~2, 391--412. \MR{4765775}

\bibitem[LSZ22]{LSZ17}
\textsc{D.~Loeffler, C.~Skinner,} and \textsc{S.~L. Zerbes},
  \articlehref{https://doi.org/10.4171/jems/1124}{\emph{Euler systems for
  {${\rm GSp}(4)$}}}, J. Eur. Math. Soc. (JEMS) \textbf{24} (2022), no.~2,
  669--733. \MR{4382481}

\bibitem[LZ20a]{LZ20b-regulator}
\textsc{D.~Loeffler} and \textsc{S.~L. Zerbes}, \emph{On {$p$}-adic regulators
  for {$\GSp(4)\times \GL(2)$} and {$\GSp(4)\times \GL(2) \times \GL(2)$}},
  preprint, 2020.

\bibitem[LZ20b]{LZvista}
\bysame, \emph{$p$-adic {L}-functions and diagonal cycles for
  $\operatorname{GSp}(4)\times\operatorname{GL}(2)\times\operatorname{GL}(2)$},
  in preparation, 2020.

\bibitem[LZ21]{LZ21-erl}
\bysame, \articlehref{http://arxiv.org/abs/2106.14511}{\emph{{\noopsort{A}}{O}n
  the {B}loch--{K}ato conjecture for
  {$\operatorname{GSp}(4)\times\operatorname{GL}(2)$}}}, preprint, 2021,
  \path{arXiv:2106.14511}.

\bibitem[Oka19]{okazaki}
\textsc{T.~Okazaki}, \articlehref{http://arxiv.org/abs/1902.07801}{\emph{Local
  {W}hittaker-newforms for {$\operatorname{GSp}(4)$} matching to {L}anglands
  parameters}}, preprint, 2019, \path{arXiv:1902.07801}.

\end{thebibliography}
\end{document}